\documentclass[reqno, 11pt]{amsart}
\usepackage{amsmath,amssymb,graphicx,color,mathrsfs,tikz-cd,amsthm, mathtools, quiver, enumerate}
\usepackage[skip=10pt plus0pt, indent=40pt]{parskip}
\usepackage{hyperref} 
\hypersetup{
    colorlinks=true,
    linkcolor=red,
    filecolor=purple,      
    urlcolor=brown,
    pdftitle={Overleaf Example},
    pdfpagemode=FullScreen,
    }
\usepackage{spalign}
\makeatletter
\def\@biblabel#1{}
\renewcommand\@cite[2]{{#1\if@tempswa,\nolinebreak[3] #2\fi}}
\makeatother
\usepackage[bmargin = 50pt, tmargin = 50pt, right=50pt, left=50pt]{geometry}
\usepackage{enumerate}
\newcommand\NN{\mathbb{N}}
\newcommand\OO{\mathscr{O}}
\newcommand\QQ{\mathbb{Q}}

\newcommand\ZZ{\mathbb{Z}}

\newcommand\pp{\mathfrak{p}}
\newcommand\minus{\smallsetminus}

\newcommand\FF{\mathbb{F}}

\theoremstyle{plain}
\newtheorem{thm}{Theorem}[section]
\newtheorem{cor}[thm]{Corollary}
\newtheorem{lem}[thm]{Lemma}
\newtheorem{prop}[thm]{Proposition}
\newtheorem{conj}[thm]{Conjecture}

\theoremstyle{definition}
\newtheorem{defn}[thm]{Definition}

\theoremstyle{remark}
\newtheorem{rmk}[thm]{Remark}
\newtheorem{notn}[thm]{Notation}

\theoremstyle{remark}

\newtheorem{schk}[thm]{Sanity Check}

\theoremstyle{remark}
\newtheorem{intuition}[thm]{Intuition}

\theoremstyle{remark}

\DeclareMathOperator{\HHom}{\ensuremath{\mathcal{H}\hspace{-.25ex}\mathit{om}}}
\DeclareMathOperator{\Hom}{Hom}

\renewcommand\mod[1]{\ (\mathrm{mod}\ #1)}

\newcommand\qq{{\mathfrak{q}}}

\newcommand{\spec}{\operatorname{Spec}}

\newcommand\colim{\operatorname{colim}}

\newcommand\AAa{\mathbb{A}}

\newcommand\pic{\operatorname{Pic}}

\newcommand{\et}{{\operatorname{\acute et}}}

\newcommand{\mor}{\operatorname{Mor}}
\newcommand\GG{\mathbb{G}}
\newcommand\perf{{\operatorname{perf}}}
\newcommand\ptp{{\neq p}}
\newcommand\plim[1]{\underset{#1}\varprojlim\:}
\newcommand\otop{\mathscr{O}^*_{\operatorname{top}}}
\newcommand\ohat{\widehat{\mathscr{O}}^*(\overline{X})}
\newcommand\xbar{\overline{X}}
\newcommand\hone{H^1_\et(\xbar,\widehat\ZZ(1))}

\renewcommand{\top}{{\operatorname{top}}}
\newcommand\cts{{\operatorname{cts}}}
\newcommand\pix{\pi_1^\et(X,\overline{x})}

\newcommand\pixbar{\pi_1^\et(\xbar,\overline{x})}

\newcommand\pigmbar{\pi_1^\et(\GG_{m,\overline{K}},1)}
\newcommand\pigmtame{\pi_1^\et(\GG_{m,\overline{K}},1)^{t}}
\newcommand\sep{{\operatorname{sep}}}
\newcommand\kbar{\overline{K}}

\usepackage{xpatch}
\makeatletter   
\xpatchcmd{\@tocline}
{\hfil\hbox to\@pnumwidth{\@tocpagenum{#7}}\par}
{\ifnum#1<0\hfill\else\dotfill\fi\hbox to\@pnumwidth{\@tocpagenum{#7}}\par}
{}{}
\makeatother    

\makeatletter
 \def\l@subsection{\@tocline{2}{0pt}{3pc}{6pc}{}}
\def\l@subsubsection{\@tocline{3}{0pt}{8pc}{8pc}{}}
 \makeatother

\newcounter{jdrthmtype}

\newenvironment{oneoffthm*}[2][plain]{%
  \addtocounter{jdrthmtype}{1}%
  \theoremstyle{#1}%
  \newtheorem*{oneoff\thejdrthmtype}{#2}%
  \begin{oneoff\thejdrthmtype}}{%
  \end{oneoff\thejdrthmtype}}

\title{\'Etale Reconstruction for $\FF_p(t)$-Schemes}
\author{Zachary Berens}
\date{}

\begin{document}

\begin{abstract}
    Voevodsky proved that normal schemes of finite type over finitely generated fields of characteristic $0$ can be reconstructed from their \'etale sites. Let $K$ be a field that is finitely generated over $\FF_p(t)$. Grothendieck conjectured that perfections of finite type $K$-schemes can be reconstructed from their \'etale sites. Adapting Voevodsky's methods, we prove this.
    \end{abstract}
\maketitle
\tableofcontents
\pagebreak
\section{Introduction}
Let $k$ be an absolutely finitely generated (AFG) field of characteristic $0$, that is, a field that is finitely generated over $\QQ$. Let $X$ and $Y$ be finite type $k$-schemes. Let $\operatorname{Shv}(X_\et)$ and $\operatorname{Shv}(Y_\et)$ be the \'etale topoi of $X$ and $Y$ respectively. In his letter to Faltings [\href{https://webusers.imj-prg.fr/~leila.schneps/grothendieckcircle/Letters/GtoF.pdf}{LTF}], Grothendieck conjectured that, if $X$ is normal, there is a bijection $$\mor_k(X,Y)\longrightarrow \mor_{\operatorname{Shv}(k_\et)}(\operatorname{Shv}(X_\et), \operatorname{Shv}(Y_\et))$$
where the right-hand side denotes the set of (2)-isomorphism classes of $\operatorname{Shv}(k_\et)$-morphisms of \'etale topoi, and the map sends a morphism of schemes to the induced geometric morphism of topoi given by pullback.\footnote{Grothendieck mentions that there can be at most one isomorphism between any two $\operatorname{Shv}(k_\et)$-morphisms of sites. This rigidity property implies that $\mor_{\operatorname{Shv}(k_\et)}(\operatorname{Shv}(X_\et), \operatorname{Shv}(Y_\et))$ is a set so that we can say ``bijection." I do not know of a proof of this, and I do not know if it is true. But [\cite{CHW}, \href{https://arxiv.org/pdf/2407.19920\#page=13}{2.22 Proposition}] shows it for admissible morphisms (morphisms of topoi for which the induced morphism on underlying topological spaces sends closed points to closed points) which is all one needs to connect Voevodsky's result to Grothendieck's conjecture.} 

Let $X$ be a scheme and let $X_\et$ denote its quasicompact and separated \'etale site whose opens are \'etale, quasicompact, and separated morphisms $U\longrightarrow X$. Voevodsky proved the following:\footnote{Voevodsky mistakenly assumes that $X$ only needs to be nonsingular in codimension $1$. See Remark \ref{r1}.} 
\begin{thm}[\cite{Voe}, \href{https://www.math.ias.edu/vladimir/sites/math.ias.edu.vladimir/files/Etale_topologies_published.pdf\#page=10}{Theorem 3.1}]\label{voethm}
    If $X$ and $Y$ are $k$-schemes of finite type and $X$ is normal, there is a bijection $$\mor_k(X,Y)\longrightarrow\mor_{k_\et}^\bullet(X_\et,Y_\et)$$ where the right-hand side is the set of \emph{($2$-)}isomorphism classes of ``admissible" $\spec(k)_\et$-morphisms from $X_\et$ to $Y_\et$.
\end{thm}
As before, the map is given by pullback. An admissible morphism is a morphisms of sites whose induced map on topological spaces sends closed points to closed points (Definition \ref{admissible}; there is a parallel notion of admissibility for geometric morphisms of topoi). An immediate corollary of Voevodsky's result is that two normal schemes $X$ and $Y$ of finite type over $k$ are isomorphic over $k$ if and only if there is a $\spec(k_\et)$-equivalence of sites $X_\et\simeq Y_\et$ [\cite{Voe}, \href{https://www.math.ias.edu/vladimir/sites/math.ias.edu.vladimir/files/Etale_topologies_published.pdf\#page=10}{Corollary 3.1}]. In particular, a normal scheme can be reconstructed from its quasicompact and separated \'etale site. It turns out that morphisms of quasicompact and separated sites are the same as morphisms of topoi in this setting [\cite{CHW}, \href{https://arxiv.org/pdf/2407.19920#page=35}{A.13 Corollary}]. And the topos of sheaves on $X_\et$ is equivalent to the usual \'etale topos if $X$ is qs [\cite{CHW}, \href{https://arxiv.org/pdf/2407.19920#page=35}{A.4 Lemma}]. So Voevodsky showed that Grothendieck's conjecture is true after restricting to admissible morphisms.\footnote{It should be noted that Voevodsky says that his result constitutes a proof of Grothendieck's conjecture. He just did not write down a proof of the equivalence between morphisms of sites and morphisms of topoi.}

Let $K$ be an absolutely finitely generated field of any characteristic, so $K$ is a field that is finitely generated over its prime field. In \emph{Esquisse d'un Programme} [\href{https://webusers.imj-prg.fr/~leila.schneps/grothendieckcircle/EsquisseEng.pdf}{EDP}], Grothendieck made a more general conjecture (that is also intimated in his Letter to Faltings). Following Voevodsky, we have to modify the conjecture by restricting to admissible morphisms.
 \begin{conj}[The Toposic Hom Conjecture]
      Let $\mathscr{C}=\textbf{\emph{Sch}}_{K}^{\operatorname{ft}}[\operatorname{UH}^{-1}]$ be the localization of the category of finite type $K$-schemes at the universal homeomorphisms. Let $X,Y\in\operatorname{Ob}(\mathscr{C})$. Then the natural map
      $$\mor_{\mathscr{C}}(X,Y)\longrightarrow\mor_{\operatorname{Shv}(K_\et)}^\bullet(\operatorname{Shv}(X_\et),\operatorname{Shv}(Y_\et))$$
      is a bijection.
 \end{conj}
This is essentially the hom form of the anabelian main conjecture with \'etale topoi instead of \'etale fundamental groups. For this reason, it makes sense to call it the toposic hom conjecture. The \'etale fundamental group factors through the \'etale topos so that we have functors
$$\textbf{Sch}_K^{\operatorname{ft}}\longrightarrow \textbf{Topoi}_K\longrightarrow\textbf{Groups}$$ The hom conjecture cares about the situations in which the composite is fully faithful. We are interested in the situations in which the first functor is fully faithful. The idea is that, in nice cases, universal homeomorphisms should be the only morphisms that the \'etale topos can't see. A universal homeomorphism $X\longrightarrow Y$ is a map for which $|X\times_YZ|\cong|Z|$ for all $Z\longrightarrow Y$. Equivalently, a universal homeomorphism is integral, surjective, and universally injective [\cite{Stacks}, \href{https://stacks.math.columbia.edu/tag/04DF}{Tag 04DF}]. Inverting universal homeomorphisms is necessary because the \'etale topology cannot distinguish between schemes that are universally homeomorphic---this is the ``topological invariance of the \'etale site" [\cite{Stacks}, \href{https://stacks.math.columbia.edu/tag/04DY}{Tag 04DY}]. In particular, if $f:X\longrightarrow Y$ is a universal homeomorphism, the induced morphism between \'etale sites given by pullback is an equivalence. 

A scheme is absolutely weakly normal if any universal homeomorphism into it is an isomorphism. Any scheme $X$ admits a canonical absolute weak normalization $X^{\operatorname{awn}}\longrightarrow X$ which is the initial object in the category of universal homeomorphisms $Y\longrightarrow X$ [\cite{Stacks}, \href{https://stacks.math.columbia.edu/tag/0EUK}{Tag 0EUK}]. Over $\QQ$, the absolute weak normalization is the seminormalization [\cite{Rydh}, \href{https://arxiv.org/pdf/0710.2488#page=40}{ Remark B.8}] (see Definition \ref{sn}; Intuition \ref{intuition}). Over $\FF_p$, it is the perfection (the limit over the absolute Frobenius) [\cite{BS}, \href{https://arxiv.org/pdf/1507.06490#page=11}{Lemma 3.8/Remark 3.9}]. Via the functor that sends a scheme to its absolute weak normalization, the localization of the category of finite type $K$-schemes at the universal homeomorphisms is equivalent to the full subcategory of $K^{\operatorname{awn}}$-schemes that are absolute weak normalizations of a finite type $K$-scheme [\cite{BGH}, \href{https://arxiv.org/pdf/1807.03281.pdf#page=229}{14.3.3 Corollary}]. Since seminormalization preserves being of finite type, and since $\spec(K)$ is seminormal, this localization is equivalent to the category of seminormal $K$-schemes when $\operatorname{char}(K)=0$. The perfection of a scheme of finite type, on the other hand, is almost never of finite type. So when $\operatorname{char}(K)>0$, we work with the category of perfect $K^{\operatorname{awn}}=K^{p^{-\infty}}-$schemes that are universally homeomorphic to a finite type $K$-scheme.

At the beginning of [\cite{Voe}], Voevodsky suggests that, ``with some modifications," his proof ``seems to apply to schemes over finitely generated fields of characteristic $p>0$ [that] have transcendence degree $\geq1$." He was right when it comes to perfections of finite type schemes. Let $K$ be an infinite AFG field.
\begin{oneoffthm*}[plain]{Theorem \ref{main}}
    Let $X=X_0^\perf$ be the perfection of a finite type $K$-scheme and let $Y$ be any finite type $K^{p^{-\infty}}-$scheme. Then the natural map
    $$\mor_{K^{p^{-\infty}}}(X,Y)\longrightarrow\mor_{K_\et}^\bullet(X_\et,Y_\et)$$
    is a bijection.
\end{oneoffthm*}
\begin{oneoffthm*}{Corollary \ref{cor}}
    Let $X$ and $Y$ be perfections of $K$-schemes of finite type. Then the natural map $$\mor_{K^{p^{-\infty}}}(X,Y)\longrightarrow\mor_{K_\et}^\bullet(X_\et,Y_\et)$$
    is bijective.
\end{oneoffthm*}
\begin{oneoffthm*}{Corollary \ref{coriso}}
    Let $X$ and $Y$ be perfections of schemes of finite type over $K$. Then $X\cong Y$ if and only if $X_\et\simeq Y_\et$ (over $K^{p^{-\infty}}$ and $\spec(K)_\et$ respectively).
\end{oneoffthm*}

This settles the toposic hom conjecture in positive characteristic for all but the case of finite fields.
While writing this up, I learned that Carlson, Haine, and Wolf independently show that Voevodsky's methods can be extended to prove reconstruction over infinite, AFG fields of positive characteristic [\cite{CHW}, \href{https://arxiv.org/pdf/2407.19920#page=34}{6.25 Corollary}]. They also show that seminormal schemes of finite type over an AFG field of characteristic $0$ can be reconstructed from their \'etale sites; in particular, the toposic hom conjecture is true over AFG fields of characteristic $0$ [\cite{CHW}, \href{https://arxiv.org/pdf/2407.19920#page=2}{0.3 Theorem}] (see Remark \ref{hdesc}). More recently, Carlson and Stix proved reconstruction for seminormal schemes of finite type over sub $p$-adic fields [\cite{CS}, \href{https://arxiv.org/pdf/2410.22474#page=10}{Theorem 4.1}].
\begin{rmk}\label{r1} Let $k$ be an AFG field of characteristic $0$. For [\cite{Voe}, \href{https://www.math.ias.edu/vladimir/sites/math.ias.edu.vladimir/files/Etale_topologies_published.pdf#page=10}{Theorem, 3.1}], Voevodsky actually only assumes that $X$ is nonsingular in codimension $1$ ($R_1$). This is not quite sufficient. In his proof, he uses a lemma that says that Picard groups of $R_1$ schemes of finite type over $k$ are finitely generated. But the proof of this lemma works only if $X$ is $S_2$ in addition to being $R_1$. There are $R_1$ schemes of finite type over $k$ that fail to be \'etale reconstructible. For example, let $X$ be the scheme obtained from $\AAa^2_k$ by collapsing two linearly independent tangent vectors, creating a cusp. Because of the cusp, $X$ is not seminormal. So it is not normal. But it is $R_1$, so it is not $S_2$. Its seminormalization $\AAa^2_k\longrightarrow X$ induces an equivalence on \'etale sites $\AAa^2_{k,\et}\overset{\sim}{\longrightarrow}X_\et$ whose inverse does not come from a morphism of schemes. So there exists an \'etale open $U$ of $X$ such that $\pic(U)$ has infinitely divisible torsion elements. We can show that the Picard group of $X=\spec(A)$ is not finitely generated. Consider the Milnor pullback square
$$\begin{tikzcd}
A \arrow[r, "f"] \arrow[d, "\iota"'] & {k[x,y]} \arrow[d, "\Tilde{\iota}"]      \\
{k} \arrow[r, "\overline{f}"]    & {k[x,y]/(x,y)^2}
\end{tikzcd}$$
(see [\cite{Wei}, \href{https://sites.math.rutgers.edu/~weibel/Kbook/Kbook.pdf#page=21}{Example 2.6}] for the definition). A little computation shows that $$A=k[x^2,xy,y^2,x^3,y^3, xy^2, yx^2]$$ By [\cite{Wei}, \href{https://sites.math.rutgers.edu/~weibel/Kbook/Kbook.pdf#page=33}{Theorem 3.10}], we get a Mayer-Vietoris sequence
$$1\rightarrow A^\times\rightarrow (A/(x,y)^2)^\times\times k[x,y]^\times\rightarrow (k[x,y]/(x,y)^2)^\times\rightarrow\pic(X)\rightarrow \pic(\AAa^2_k)\times\pic(k)\rightarrow\ldots$$
which simplifies to
$$1\longrightarrow k^\times\longrightarrow k^\times\times k^\times\longrightarrow (k[x,y]/(x,y)^2)^\times\longrightarrow\pic(X)\longrightarrow1$$
where $k^{\times}\longrightarrow k^{\times}\times k^{\times}$ is the diagonal, and $$k^{\times}\times k^{\times}\longrightarrow (k[x,y]/(x,y)^2)^\times=\{a+bx+cy:(a,b,c)\in k^\times\times k\times k\}$$ is given by $$(a,b)\mapsto \Tilde{\iota}(b)\overline{f}(a)^{-1}=ba^{-1}$$ So $$\pic(X)\cong (k[x,y]/(x,y)^2)^\times/k^\times\cong\{1+ax+by:a,b\in k\}\cong k^+\times k^+$$ (c.f. Sanity Check \ref{sans}). Intuitively, the two $k^+$'s correspond to the two tangent vectors we collapsed to create $X$.
\end{rmk}
\textbf{Acknowledgements.} I would like to thank Joe Rabinoff for being an excellent mentor and teacher throughout the course of this project and my time at Duke. I would also like to thank Kirsten Wickelgren for suggesting that I try to generalize Voevodsky's theorem. Finally, I would like to thank Peter Haine for being so encouraging of the project. These results were obtained while receiving funding from the Duke math department's PRUV program in the summer of 2024.
\subsection{Characteristic $0$}
Let $k$ be an AFG field of characteristic $0$. As mentioned before, absolute weak normality is the same as seminormality when we are working over $k$. There is a more concrete and geometric criterion for a scheme to be seminormal. The idea is that we want to avoid having cusp-like singularities. \begin{defn}[\cite{Stacks}, \href{https://stacks.math.columbia.edu/tag/0EUK}{Tag 0EUK}]\label{sn}
    An affine $k$-scheme $X=\spec(A)$ is \emph{seminormal} if for all $x,y\in A$ such that $x^2=y^3$, there exists an $a\in A$ such that $a^2=y$ and $a^3=x$.
\end{defn}
One immediately sees from this definition that the cuspidal cubic $\spec k[x,y]/(y^2-x^3)$ is not seminormal.
\begin{intuition}[\cite{SS}]\label{intuition}
    Seminormal schemes are schemes for which all non-normality comes from maximally transverse gluing. Any non-normal affine scheme $X$ of finite type over $k$ can be constructed as the pushout of a diagram $$\begin{tikzcd}
Y \arrow[r, "\text{finite}", two heads] \arrow[d, "\text{closed}"', hook] & Z \\
\widetilde{X}                                               &  
\end{tikzcd}$$ where $\widetilde{X}$ is the normalization of $X$, $Z$ is the non-normal locus of $X$, and $Y$ is the inverse image of the non-normal locus in $\widetilde{X}$. If $X$ is seminormal, $Y$ is reduced. So there is no ``undue gluing of tangent spaces" in the non-normal locus of a seminormal scheme. The cuspidal cubic is not seminormal because it is given by gluing $\widetilde{X}=\AAa^1$ to $Z=\spec(k)$ along $Y=\spec(k[\varepsilon]/\varepsilon^2)$. The nodal cubic, on the other hand, \emph{is} seminormal because it is given by gluing $\widetilde{X}=\AAa^1$ to $Z=\spec(k)$ along $Y=\spec(k)\coprod\spec(k)$.
\end{intuition}
Let $X$ and $Y$ be finite type $k$-schemes with $X$ normal. Consider morphisms $\mor_k(X,Y)$. Voevodsky's proof of Theorem \ref{voethm} proceeds in three steps:
\begin{enumerate}
    \item Show that it suffices to (a) replace $Y$ with $\GG_m$ and (b) show that, for every affine, geometrically connected scheme $U$ that is \'etale over $X$ and every $\varphi\in\mor_{k_\et}^\bullet(U_\et,\GG_{m,\et})$, there exists a map of $k$-schemes $U\longrightarrow\GG_m$ that coincides with $\varphi$ on geometric points ($\varphi$ is admissible so it gives a map on geometric points: see Proposition \ref{injprep}).
    \item We may replace $U$ with $X$ because schemes \'etale over a normal scheme are also normal. We construct maps
    $$\begin{tikzcd}
\ohat  \arrow[rd, "\chi"]          &       \\
                                                                         & \hone \\
{ \otop(X)} \arrow[ru, "\chi_\top"'] &      
\end{tikzcd}$$ where \begin{align*}
    \ohat&=\plim{n}\OO^*(\xbar)/\OO^*(\xbar)^n \\
    \otop(X)&=\mor^\bullet_{k_\et}(X_\et,\GG_{m,\et})
\end{align*}and show that it suffices for the image of $\chi_\top$ to lie in the image of $\chi$.
    \item Show that the image of $\chi_\top$ lies in the image of $\chi$ when $X$ is normal by showing that the Picard group of a normal $k$-scheme has bounded torsion i.e., that the Tate module $\widehat{T}(\pic(X))$ vanishes.\footnote{Voevodsky makes a typo when he says that $\widehat{T}(\pic(X))=0$ is equivalent to saying that $\pic(X)$ has no infinitely divisible elements. He means no infinitely divisible \emph{torsion} elements.} This follows from Mordell-Weil-N\'eron-Lang which implies that $\pic(X)$ is finitely generated. This is the only point at which normality is used.
\end{enumerate}
\begin{schk}[The Picard Group of the Once-Punctured Cuspidal Cubic]\label{sans}
    Regarding (3), it should be noted that there are non-seminormal schemes with Picard groups that have bounded torsion. For example, the Picard group of the cuspidal cubic over $k$ is $k^+$. For reconstruction, one would need to show that the Picard groups of all geometrically connected affine schemes that are \'etale over the cuspidal cubic have bounded torsion. But there are \'etale opens of the cuspidal cubic with Picard groups of unbounded torsion. An example is the once-punctured cuspidal cupic $X$ (with a puncture at $(1,1)$, away from the cusp). After some computation, one finds that \begin{align*}
        X&=\spec\big(k[x,y,\frac{1}{x-1}]/(y^2-x^3)\cap k[x,y,\frac{1}{y-1}]/(y^2-x^3)\big)\\&=\spec\big(k[x,y,\frac{y+1}{x-1}]/(y^2-x^3)\big)\\&=\spec k[t^2,t^3,\frac{t^3+1}{t^2-1}]
    \end{align*} This is a non-distinguished affine open of the cuspidal cubic. The seminormalization of $X$ is $\GG_m$. So the natural map $$\mor_k(X,\GG_m)\longrightarrow\mor_{k_\et}^\bullet(X_\et,\GG_{m,\et})$$
    can't be surjective because the equivalence $X_\et\overset{\sim}{\longrightarrow}\GG_{m,\et}$, given by the inverse of the equivalence induced by the seminormalization, is not realized by an isomorphism of schemes. One can compute the Picard group of $X$ as follows: Consider the Milnor square $$\begin{tikzcd}
{k\left[t^2,t^3,\frac{t^3+1}{t^2-1}\right]} \arrow[r] \arrow[d] & {k\big[t,\frac{1}{t-1}\big]} \arrow[d]    \\
k \arrow[r]                                                     & {k[\varepsilon]/\varepsilon^2}
\end{tikzcd}$$
(see [\cite{Wei}, \href{https://sites.math.rutgers.edu/~weibel/Kbook/Kbook.pdf#page=21}{Example 2.6}] for the definition of a Milnor square). The top map records the seminormalization of $X$ (which is the same as the normalization) and the map on the right is given by $(t-1)^n\mapsto (-1)^n(1-n\varepsilon)$. Note that $k[t^2,t^3,\frac{t^3+1}{t^2-1}]$ is a subring of $k\big[t,\frac{1}{t-1}\big]$ because $\frac{t^3+1}{t^2-1}=t+\frac{1}{t-1}$. This pullback square is Milnor with respect to the non-normal locus of $X$: $k[t^2,t^3,\frac{t^3+1}{t^2-1}]/(t^2,t^3)=k$ is the ring of functions of the non-normal locus of $X$, the ideal $(t^2,t^3)$ maps isomorphically to $(t^2)$ in $k[t,\frac{1}{t-1}]$, and we have $k[t,\frac{1}{t-1}]/(t^2)\cong k[\varepsilon]/\varepsilon^2$ via the map on the right. By [\cite{Wei}, \href{https://sites.math.rutgers.edu/~weibel/Kbook/Kbook.pdf#page=33}{Theorem 3.10}], we get a Mayer-Vietoris sequence
$$1\longrightarrow k^\times\longrightarrow k\big[t,\frac{1}{t-1}\big]^\times\times k^\times\longrightarrow (k[\varepsilon]/\varepsilon^2)^\times\longrightarrow\pic(X)\longrightarrow\pic(\GG_m)\times\pic(k)\longrightarrow\ldots$$
We have $(k[\varepsilon]/\varepsilon^2)^\times\cong\{a+b\varepsilon:a\in k^\times,b\in k\}$,
so our exact sequence becomes 
$$1\longrightarrow k^\times\longrightarrow k[(t-1)^{\pm1}]^\times\times k^\times\longrightarrow (k[\varepsilon]/\varepsilon^2)^\times\longrightarrow \pic(X)\longrightarrow 1$$
where $k^\times\longrightarrow k[(t-1)^{\pm1}]^\times\times k^\times$ is the diagonal map and $k[(t-1)^{\pm1}]^\times\times k^\times\longrightarrow k[\varepsilon]/\varepsilon^2$ is given by $$(a(t-1)^n,b)\longmapsto b\big(a(-1)^n(1-n\varepsilon)\big)^{-1}=ba^{-1}(-1)^n(1+n\varepsilon)$$ 
(the diagonal copy of $k^\times$ is indeed the kernel of this map). We have $(k[\varepsilon]/\varepsilon^2)^\times/k^\times\cong\{1+c\varepsilon:c\in k\}\cong k^+$ and $\{1+n\varepsilon:n\in\ZZ\}\cong \ZZ< k^+$. So $\pic(X)\cong k^+/\ZZ$. This group evidently has infinitely divisible torsion elements.
\end{schk}
Normal schemes are seminormal, but there are non-normal seminormal schemes e.g., the nodal cubic. It turns out that normality is indeed a suboptimal assumption for reconstructing schemes from their \'etale sites. Voevodsky's proof gives us a method for proving generalizations in characteristic $0$. We start with a scheme $X$ satisfying some property $P$. Then we show that, if $U\longrightarrow X$ is \'etale, $U$ satisfies $P$. Then we show that the Picard groups of $P$ schemes have bounded torsion. Using this strategy and results of [\cite{GJRW}], we will conclude that seminormal, $S_2$ schemes of finite type over $k$ are \'etale reconstructible.
\begin{lem}[\cite{Rush}, \href{https://www.sciencedirect.com/science/article/pii/0022404982900329}{Lemma 1.10}]
    If $U\longrightarrow X$ is \'etale and $X$ is a seminormal $k$-scheme of finite type, then $U$ is also seminormal.
\end{lem}
The following two lemmas are well known. We give proofs for completeness.
\begin{lem}
    Being $S_2$ is an \'etale-local property.
\end{lem}
\begin{proof}
    Let $X=\spec(A)$. Let $U=\spec(B)\longrightarrow X$ be \'etale. Take $\qq\in\spec(B)$ lying above $\pp\in\spec(A)$. Then $A_\pp\longrightarrow B_\qq$ is an \'etale map of local rings, so it is faithfully flat. Hence all primes in $B$ of height $\geq 2$ lie above the primes in $A$ of height $\geq 2$. Since $A_\pp\longrightarrow B_\qq$ is \'etale, $\operatorname{depth}(A_\pp)=\operatorname{depth}(B_\qq)$. Conversely, if $U\longrightarrow X$ is \'etale and surjective and $U$ is $S_2$ then $X$ is $S_2$ by the same argument.
\end{proof}

\begin{lem}\label{pics2}
    Let $X$ be a Noetherian $S_2$ scheme. Let $Z\subset X$ be a closed subscheme of codimension $\geq2$. Then we get an injection
    $\pic(X)\hookrightarrow\pic(X\minus Z)$.
\end{lem}
\begin{proof}
    By the cohomological characterization of $S_2$, we have $H^0_Z(X,\mathscr{F})=H^1_Z(X,\mathscr{F})=0$ for any coherent sheaf $\mathscr{F}$ with support on $X$ (in general, a scheme is $S_n$ if and only if $H^i_Z(X,\mathscr{F})=0$ for all $i<n$ for all closed $Z\subset X$ of codimension $\geq n$). So the long exact sequence in cohomology with closed support $Z$ $$H^0_Z(X,\mathscr{F})\longrightarrow H^0(X,\mathscr{F})\longrightarrow H^0(X\minus Z,\mathscr{F})\longrightarrow H^1_Z(X,\mathscr{F})\longrightarrow\ldots$$
    yields an isomorphism
    $$\operatorname{res}_{X\minus Z}(\mathscr{F}):H^0(X,\mathscr{F})\overset{\sim}{\longrightarrow} H^0(X\minus Z,\mathscr{F})$$
    Let $\mathscr{L}\in\pic(X)$ be a line bundle. Let $\mathscr{H}=\HHom(\mathscr{L},\OO_X)$ and $\mathscr{H}'=\HHom(\OO_X,\mathscr{L})$. Assume that $\mathscr{L}|_{X\minus Z}\cong \OO_{X\minus Z}$ and choose an isomorphism $s\in H^0(X\minus Z,\mathscr{H})$. Its inverse $s^{-1}$ is in $H^0(X\minus Z,\mathscr{H}')$. So $s\circ s^{-1}=\operatorname{id}_{\mathscr{L}|_{X\minus Z}}$ and $s^{-1}\circ s=\operatorname{id}_{\OO_{X\minus Z}}$. Since $\operatorname{res}_{X\minus Z}(\mathscr{H})$ and $\operatorname{res}_{X\minus Z}(\mathscr{H}')$ are isomorphisms, we get an isomorphism $\mathscr{L}\cong \OO_X$ from the inverse images of $s$ and $s^{-1}$ in $H^0(X,\mathscr{H})$ and $H^0(X,\mathscr{H}')$.
\end{proof}
\begin{thm}\label{sns2}
Let $X$ be a seminormal, $S_2$ scheme of finite type over $k$. Let $Y$ be any $k$-scheme of finite type. The map 
$$\mor_k(X,Y)\longrightarrow\mor^\bullet_{k_\et}(X_\et,Y_\et)$$
is a bijection.
    
\end{thm}

\begin{proof}[Proof]
    It suffices to show that $\pic(X)$ has bounded torsion. As mentioned above in (1), we may reduce to the case in which $X$ is affine. We reproduce the argument given in [\cite{GJRW}, \href{https://arxiv.org/pdf/alg-geom/9410031\#page=25}{Theorem 6.5}] which shows that $$\pic(X)\cong \text{finite group}\oplus\text{countably generated free abelian group}$$ Let $\mathcal{C}$ be the class of groups that can be written as the direct sum of a finite group and a countably generated free abelian group. Let $X=\spec(A)$ and let $\widetilde{X}=\spec(\widetilde{A})$ be the normalization of $X$. Let $I=\{x\in A:x\widetilde{A}\subset A\}$ be the conductor of $A$ in $\widetilde{A}$ (which is an ideal of both $A$ and $\widetilde{A}$); the non-normal locus of $X$ is $Z=\spec(A/I)$. The non-normal locus $Q$ of $Z$ is of codimension $\geq 2$ in $X$ (because $Q$ is of codimension $\geq 1$ in $Z$, and $Z$ is of codimension $\geq 1$ in $X$). It is also closed in $X$. So by Lemma \ref{pics2}, $\pic(X)\hookrightarrow\pic(X\minus Q)$. Hence, we may replace $X$ with $X\minus Q$ and proceed under the assumption that $Z$ is normal. In particular, $\pic(Z)$ is finitely generated. We have a pushout square
    $$
      \begin{tikzcd}
Y \arrow[r, "\text{finite}", two heads] \arrow[d, "\text{closed}"', hook] & Z \arrow[d] \\
\widetilde{X} \arrow[r]                                                   & X          
\end{tikzcd}
    $$
(see [\cite{Wei}, \href{https://sites.math.rutgers.edu/~weibel/Kbook/Kbook.pdf#page=21}{Example 2.6}]) where $Y$ is the inverse image of $Z$ in $\widetilde{X}$. Since $X$ is seminormal, $Y$ is reduced. The pushout is a Milnor square, so we get a Mayer–Vietoris sequence 
$$\ldots\longrightarrow\widetilde{A}^\times\oplus(A/I)^\times\longrightarrow(\widetilde{A}/I)^\times\longrightarrow\pic(X)\longrightarrow\pic(\widetilde X)\oplus\pic(Z)\longrightarrow\pic(Y)\longrightarrow\ldots$$
Let $\Lambda$ denote the integral closure of $k$ in $\widetilde{A}/I$. Let $E_1$ denote the integral closure of $k$ in $A/I$. And let $E_2$ denote the image in $\Lambda$ of the integral closure of $k$ in $A$. Let $D$ denote the kernel of the map $\pic(X)\longrightarrow \pic(\widetilde X)\oplus\pic(Z)$. Since $Y$ is reduced, [\cite{GJRW}, \href{https://arxiv.org/pdf/alg-geom/9410031#page=4}{Proposition 1.6}] gives us the finitely generated term in the exact sequence
$$\text{finitely generated}\longrightarrow \Lambda^*/E_1^*E_2^*\longrightarrow D\longrightarrow 0.$$
By [\cite{GJRW}, \href{https://arxiv.org/pdf/alg-geom/9410031#page=25}{Theorem 6.4(2)}], $\Lambda^*/E_1^*E_2^*\in\mathcal{C}$ because $Y$ is reduced. Hence, $D\in\mathcal{C}$. The Picard groups of $\widetilde X$ and $Z$ are finitely generated because $\widetilde{X}$ and $Z$ are normal. One sees that $\mathcal{C}$ is equivalent to the class of groups given by countably many generators and finitely many relations. Then it is not hard to show that $\mathcal{C}$ is closed under extensions. So $\pic(X)\in\mathcal{C}$. We conclude that $\widehat{T}(\pic(X))=0$.
\end{proof}
\begin{rmk}\label{hdesc}
    I do not know if Picard groups of seminormal $k$-schemes have bounded torsion. They certainly do not need to be finitely generated (for example, the nodal cubic has Picard group $k^\times$). In [\cite{CHW}, \href{https://arxiv.org/pdf/2407.19920#page=19}{Section 4}], it is shown that one can prove reconstruction for seminormal schemes by using $h$-descent and de Jong's alterations to reduce to the case of regular schemes which are \'etale reconstructible by Voevodsky's theorem. In particular, the toposic hom conjecture is true over AFG fields of characteristic $0$. 
\end{rmk}
\subsection{Positive Characteristic}
Fix a prime $p$ for the rest of this paper. Recall that an $\FF_p$-scheme is perfect if the absolute Frobenius endomorphism $\Phi$---given by $f\mapsto f^p$ on functions and the identity on the underlying topological space---is an isomorphism. Any $\FF_p$-scheme $X$ admits a perfection $X^\perf:=\lim_\Phi X$.\footnote{Note that $|X^\perf|\cong|X|$.} For an affine scheme $\spec(A)$, we have $\spec(A)^\perf=\spec(A^{p^{-\infty}})$ where $A^{p^{-\infty}}=\colim_\Phi A$ is the perfect closure of $A$. As mentioned before, the perfection of an $\FF_p$-scheme is its absolute weak normalization. Therefore, by the topological invariance of the \'etale site, we have $X_\et\simeq X^\perf_\et$.\footnote{For a point $\spec(E)$ in characteristic $p$, topological invariance is equivalent to $\operatorname{Aut}(E^{\sep}/E)\cong\operatorname{Aut}(E^\sep/E^\perf)$. In particular, universal injectivity requires residue field extensions to be purely inseparable [\cite{Stacks}, \href{https://stacks.math.columbia.edu/tag/0BR5}{Tag 0BR5}]. In characteristic $0$ the fact that all fields are perfect is equivalent to the fact that a point is seminormal.} The upshot is that reconstructing perfect schemes is the best hope in positive characteristic.

The positive characteristic case differs from the characteristic $0$ case in a number of ways: \begin{itemize}
    \item We work with perfections of schemes of finite type over an infinite, AFG field of positive characteristic. These schemes are almost never of finite type. But it turns out that this does not cause any major problems because the finite type assumption is mostly used to check equality of morphisms on geometric points and the perfection does not change the underlying topological space. We need the base field to be infinite for Proposition \ref{2.11}.
    \item We have to work with prime-to-$p$ parts of every projective limit. We take $\chi={\varprojlim}_{p\,\nmid\, n}\:\chi_n$ where \hbox{$\chi_n:\OO^*(\xbar)/\OO^*(\xbar)^n\overset{}{\longrightarrow} H^1_\et(\xbar,\mu_n)$} is the coboundary map.
    \item For a base field $K$ of characteristic $0$, $\pigmbar$ is isomorphic to $\widehat\ZZ$. But in positive characteristic, we don't get $\widehat{\ZZ}_\ptp$. The problem is the existence of Artin-Schreier covers. To isolate the Kummer covers, we need to pass to the \emph{tame} fundamental group of $\GG_m$. From there we can construct the positive characteristic analogue of $\chi_\top$.
    \item Decompletion (Proposition \ref{2.11}) does not necessarily go through if $X$ is imperfect (whereas it works for any finite type scheme over an AFG field of characteristic $0$). This is because taking the perfect closure inverts $p$.
\end{itemize}
The overall strategy of the proof is to show that $\operatorname{Im}\chi_\top\subset\operatorname{Im}\chi$ suffices if $X$ is the perfection of a scheme of finite type over an infinite, AFG field of positive characteristic. \section{Preparations}
Fix a perfect field $K$ of any characteristic. We let $\Gamma_E=\operatorname{Gal}(E^\sep/E)$ for any field $E$. We will work with sites as Voevodsky did. In this section and the next, we will produce, reproduce, and fix proofs of the results in [\cite{Voe}] while adapting them to the characteristic $p$ situation. \begin{rmk}
    We work over a perfect field until the very end where we have to work with a scheme that starts life over an infinite, AFG field of positive characteristic. The absolute finite generation allows us to conclude that the Picard group has bounded torsion. Similarly, Voevodsky's use of an AFG field containing $\QQ$ is basically only used to apply Mordell-Weil-N\'eron-Lang. For this reason, reconstruction in characteristic $0$ is not limited to AFG fields. For example, one can take \hbox{$k=k_0(x_1,x_2,\ldots)$} with $k_0\supset \QQ$ AFG and Mordell-Weil-N\'eron-Lang still holds. One also needs the fact that the natural map $$\overline{k}^\times\longrightarrow \underset{E\supset k}{\varinjlim}\underset{n}{\varprojlim}\:E^\times/E^{\times n}$$ is injective, that is, $k^\times$ has no infinitely divisible elements. The divisible elements of $k$ are the same as the divisible elements of $k_0$ because $k^\times\cong k_0[x_1,x_2,\ldots]^\times\oplus\bigoplus_I \ZZ$. Since $k_0$ is AFG, it has no divisible elements. Taken with Remark \ref{hdesc}, we conclude that \'etale reconstruction goes through over $k$.
\end{rmk}
\subsection{Generalities on \'Etale Sites}
A morphism of sites $f:X\longrightarrow Y$ is a functor $f^{-1}:Y\longrightarrow X$ that sends covers to covers and preserves fiber products. The $2$-morphisms are natural transformations $f^{-1}\Rightarrow g^{-1}$. Let $\textbf{Site}$ denote the 2-category of sites. We would like to work with schemes over a base. So we have to work with sites over a base. Let $\textbf{Site}_{/S}$ denote the $2$-category of sites over a base $S$. The objects are $S$-sites. A morphism $f:X\longrightarrow Y$ is a pair $(f^{-1},\alpha)$ where $f^{-1}:Y\longrightarrow X$ is a functor as before and $\alpha:p_X^{-1}\Rightarrow f^{-1}\circ p_Y^{-1}$ is a natural isomorphism where $p_{(-)}$ denotes the structure map from $(-)$ to $S$. And a $2$-morphism $(f^{-1},\alpha_1)\Rightarrow(g^{-1},\alpha_2)$ is a natural transformation $f^{-1}\Rightarrow g^{-1}$ such that $$\begin{tikzcd}
                                                                                & f^{-1}\circ p_Y^{-1} \arrow[dd, Rightarrow] \\
p_X^{-1} \arrow[ru, "\alpha_1", Rightarrow] \arrow[rd, "\alpha_2"', Rightarrow] &                                             \\
                                                                                & g^{-1}\circ p_Y^{-1}                       
\end{tikzcd}$$ commutes.

Let $X$ be a scheme. Its quasicompact and separated \'etale site $X_\et$ has as open sets \'etale, quasicompact, and separated morphisms $U\longrightarrow X$. The coverings are the jointly surjective families. A morphism of schemes $\varphi:X\longrightarrow Y$ induces a morphism of \'etale sites $\varphi_{\et}:X_\et\longrightarrow Y_\et$ given by pullback along $\varphi$. So we get a functor $(-)_\et:\textbf{Sch}\longrightarrow \textbf{Site}$. In fact, we get a functor $(-)_\et:\textbf{Sch}_S\longrightarrow \textbf{Site}_{/S_\et}$. Indeed, take $f\in\mor_{S}(X,Y)$ and consider its \'etalification $f_\et$. To construct $\alpha_f$ to get a morphism of $S_\et$-sites, take an \'etale open $U\in S_\et$. We see that $p_X^{-1}(U)$ is the \'etale open $U\times_SX\in X_\et$. And $(f^{-1}_\et\circ p_Y^{-1})(U)$ is the \'etale open $(U\times_SY)\times_fX \in X_\et$. Then $\alpha_f(U)$ is the isomorphism $U\times_SX\overset{}{\longrightarrow}(U\times_SY)\times_fX$ given by $(\operatorname{pr}_U\times(f\circ\operatorname{pr}_X))\times \operatorname{pr}_X$ where $\operatorname{pr}_{(-)}$ denotes the projection to $(-)$.
\begin{rmk}\label{basedterminal}
    Voevodsky only assumes that morphisms of sites preserve fiber products. This does not allow us to deduce that $f^{-1}(Y)=X$ for $f=(f^{-1},\alpha):X_\et\longrightarrow Y_\et$ a morphism of \'etale sites. But when working with $\spec(K)_\et-$morphisms, we \emph{do} get preservation of terminal objects. Indeed, for $K$-schemes $X$ and $Y$, we have $$X\overset{\operatorname{id}}{\longrightarrow}X=p_X^{-1}(\spec(K)
    \overset{\operatorname{id}}{\longrightarrow}\spec(K))\cong (f^{-1}\circ p_Y^{-1})(\spec(K)
    \overset{\operatorname{id}}{\longrightarrow}\spec(K))=f^{-1}(Y\overset{\operatorname{id}}{\longrightarrow}Y)$$
\end{rmk}

Let $X$ and $Y$ be arbitrary $K$-schemes. We let $\mor_{K_\et}(X_\et,Y_\et)$ denote the set of $\spec(K)_\et$-morphisms of \'etale sites up to (2-)isomorphism.
\begin{prop}[\cite{Voe}, \href{https://www.math.ias.edu/vladimir/sites/math.ias.edu.vladimir/files/Etale_topologies_published.pdf\#page=4}{Proposition 1.2}]\label{base}
    Let $X$ and $Y$ be $K$-schemes with $Y$ of finite type. Then for any separable extension $E\supset K$ \emph{(}i.e., for any \'etale open $\spec(E)\in K_\et$\emph{)}, we get a functor 
    $$(-)_E:\mor_{K_\et}(X_\et,Y_\et)\longrightarrow\mor_{E_\et}(X_{E,\et},Y_{E,\et})$$
    that is natural relative to morphisms of $X_\et$ and $Y_\et$. And the diagram
    $$\begin{tikzcd}
{X_{E,\et}} \arrow[r, "\varphi_E"] \arrow[d, "(\operatorname{pr}_X)_\et"'] & {Y_{E,\et}} \arrow[d, "(\operatorname{pr}_Y)_\et"] \\
X_\et \arrow[r, "\varphi"]                                 & Y_\et                             
\end{tikzcd}$$
commutes.
\end{prop}

\begin{proof}
Given $\varphi=(\varphi^{-1},\alpha)$ we define $\varphi_E=(\varphi^{-1}_E,\alpha_E)$ as follows: Let $E\in K_\et$ be the \'etale open coming from $\spec{E}\longrightarrow\spec{K}$. Then $X_{E,\et}=p^{-1}_X(E)_\et$ and $Y_{E,\et}=p^{-1}_Y(E)_\et$. We take $\varphi_E^{-1}=\alpha(E)^{-1}_\et\circ\varphi^{-1}|_{Y_E}$ where $\alpha(E)_\et^{-1}$ is the functor from $\varphi^{-1}(Y_E)_\et\longrightarrow X_{E,\et}$ given by the \'etalification of $\alpha(E):X_E\overset{\sim}{\longrightarrow}\varphi^{-1}(Y_E)$. For every \'etale open $V\in E_\et$, we need to give an isomorphism $\alpha_E(V):p^{-1}_{X_E}(V)\rightarrow (\varphi^{-1}_E\circ p^{-1}_{Y_E})(V)$. Assume that $V=\spec(L)$. Then we have \begin{align*}
    (\varphi^{-1}_E\circ p_{Y_E}^{-1})(V)&=(\alpha(E)_\et^{-1}\circ\varphi^{-1}|_{Y_E}\circ p_{Y_E}^{-1})(\spec(L))\\
    &=\alpha(E)_\et^{-1}(\varphi^{-1}|_{Y_E}(Y_L))\\
    &=X_E\times_{\varphi^{-1}|_{Y_E}(Y_E)}\varphi^{-1}|_{Y_E}(Y_L)
\end{align*}
Since we have isomorphisms $\alpha(L):X_L\overset{\sim}{\longrightarrow}\varphi^{-1}|_{Y_E}(Y_L)$ and $\alpha(E):X_E\overset{\sim}{\longrightarrow}\varphi^{-1}|_{Y_E}(Y_E)$, we get
$$\begin{tikzcd}
p_{X_E}^{-1}(V)=X_L \arrow[rrd, "\sim", "\alpha(L)"', bend left] \arrow[rdd, bend right] \arrow[rd, "\sim", "\alpha_E(V)"', dashed] &                                                                                  &                                    \\
                                                                                            & X_E\times_{\varphi^{-1}|_{Y_E}(Y_E)}\varphi^{-1}|_{Y_E}(Y_L) \arrow[r, "\sim"] \arrow[d] & \varphi^{-1}|_{Y_E}(Y_L) \arrow[d] \\
                                                                                            & X_E \arrow[r, "\sim", "\alpha(E)"']                                                       & \varphi^{-1}|_{Y_E}(Y_E)          
\end{tikzcd}$$
We have found our isomorphism $\alpha_E(V)$. It is straightforward to extend to the case where $\pi_0(V)>1$.
\end{proof}
Morphisms of \'etale sites give rise to morphisms on underlying topological spaces. The idea is that, for two toplogical spaces $X$ and $Y$ with $Y$ sober, a morphism $\varphi:X\longrightarrow Y$ is uniquely determined by what it does on preimages of opens of $Y$: in particular, $$\varphi(x)=\text{Generic point of }(Y\minus \bigcup_{x\not\in\varphi^{-1}(U)}U)$$
\begin{lem}[\cite{Voe}, \href{https://www.math.ias.edu/vladimir/sites/math.ias.edu.vladimir/files/Etale_topologies_published.pdf\#page=3}{Proposition 1.1}]
    Let $\varphi:X_\et\longrightarrow Y_\et$ be a morphism of \'etale sites. There exists a unique morphism $|\varphi|:|X|\longrightarrow|Y|$ such that $|\varphi|^{-1}(U)=|\operatorname{Im}\varphi^{-1}(U)|$ for any open set $U\subset |Y|$.
\end{lem}
\begin{proof}
    Take $x\in |X|$. There is a maximal open subset $V_x\subset |Y|$ such that $|\operatorname{Im}\varphi^{-1}(V_x)|\subset |X|\minus \overline{\{x\}}$. Since $|\operatorname{Im}\varphi^{-1}(V)|$ is open, this is equivalent to saying that $V_x$ is the maximal open such that $x\not\in|\operatorname{Im}\varphi^{-1}(V_x)|$. It is given by $$V_x=\bigcup_{x\not\in|\operatorname{Im}\varphi^{-1}(U)|}U$$ It turns out that $|Y|\minus V_x$ is irreducible. To see this, let $V_x=U_1\cap U_2$ and assume that neither $U_1$ nor $U_2$ is a subset of $V_x$. Then we have $x\in |\operatorname{Im}\varphi^{-1}(U_i)|$ for $i=1,2$.  
But $U_1\cap U_2=V_x$, so, using the fact that $\varphi^{-1}$ preserves fiber products and covers, we have $$x\in|\operatorname{Im}\varphi^{-1}(U_1)|\cap|\operatorname{Im}\varphi^{-1}(U_2)|=|\operatorname{Im}\varphi^{-1}(V_x)|\subset X\minus\overline{\{x\}}$$ This is a contradiction. With this we define $$|\varphi|(x)=\text{Generic point of }\left(|Y|\minus V_x\right)$$
which we may do because $|Y|$ is sober and $|Y|\minus V_x$ is irreducible. It is then clear that $|\varphi|^{-1}(V)=|\operatorname{Im}\varphi^{-1}(V)|$ for any open subset $V\subset|Y|$.
\end{proof}
\begin{rmk}
    If $\varphi:X\longrightarrow Y$ is a morphism of schemes and $|\varphi|:|X|\longrightarrow |Y|$ is the associated map on topological spaces, then $|\varphi_\et|=|\varphi|$.   
\end{rmk}
\begin{defn}\label{admissible}
    A morphism $\varphi\in\mor(X_\et,Y_\et)$ is \emph{admissible} if $|\varphi|$ sends closed points to closed points.\footnote{In [\cite{CHW}] admissible morphisms are called \emph{pinned} morphisms. We stick with Voevodsky's terminology.}
\end{defn}
\begin{rmk}
    Admissibility is a reasonable condition because any morphism $X\longrightarrow Y$ of $K$-schemes with $X$ of finite type sends closed points to closed points. 
\end{rmk}
Given this, we will let $\mor_{K_\et}^\bullet(X_\et,Y_\et)$ denote the subset of $\mor_{K_\et}(X_\et,Y_\et)$ comprising admissible $\spec(K)_\et$-morphisms of \'etale sites up to (2-)isomorphism (similarly for $\mor^\bullet(X_\et,Y_\et)\subset \mor(X_\et,Y_\et)$).

Given a morphism $\varphi:X_\et\longrightarrow Y_\et$ and an \'etale open $U\longrightarrow Y$, we let $\varphi_U$ denote the restriction of $\varphi$ to $\varphi^{-1}(U)_\et$ i.e., $\varphi_U$ is the induced map $\varphi^{-1}(U)_\et\longrightarrow U_\et$. 
\begin{lem}[\cite{Voe}, \href{https://www.math.ias.edu/vladimir/sites/math.ias.edu.vladimir/files/Etale_topologies_published.pdf\#page=5}{Proposition 1.3}]\label{admis}
    Let $X$ and $Y$ be $\kbar$-schemes with $Y$ of finite type and pick some \hbox{$\varphi\in\mor^\bullet_{\kbar_\et}(X_\et,Y_\et)$}. Then, for any \'etale open $U\longrightarrow Y$ and any closed point $x\in X$, the map $|\varphi_U|$ is a bijection from $|\varphi^{-1}(U)_x|$ to $|U_{|\varphi|(x)}|$ where $\varphi^{-1}(U)_x$ is the fiber of $\varphi^{-1}(U)$ over $x$, and $U_{|\varphi|(x)}$ is the fiber of $U$ over $|\varphi|(x)$.
\end{lem}
\begin{proof}
To show that $|\varphi_U|$ actually maps $|\varphi^{-1}(U)_x|$ to $|U_{|\varphi|(x)}|$, it suffices to show that the following diagram commutes
$$\begin{tikzcd}
{|\varphi^{-1}(U)|} \arrow[r, "|\varphi_U|"] \arrow[d] & {|U|} \arrow[d] \\
{|X|} \arrow[r, "|\varphi|"]                           & {|Y|}          
\end{tikzcd}$$
This just follows from the functoriality of the association $X_\et\leadsto |X|, \varphi\leadsto |\varphi|$ applied to the commutative diagram $$\begin{tikzcd}
\varphi^{-1}(U)_\et \arrow[r, "\varphi_U"] \arrow[d] & U_\et \arrow[d] \\
X_\et \arrow[r, "\varphi"]                           & Y_\et          
\end{tikzcd}$$
Surjectivity: Let $y=|\varphi|(x)$ for $x\in |X|$ a closed point. Let $U\longrightarrow Y$ be \'etale. Pick a point $z\in |U_y|$ and assume that it doesn't lie in the image of $|\varphi_U|$. Let $U^z=U\minus(U_y\minus z)$. Then $|\varphi_U^{-1}(U^z)|_x=|\varphi^{-1}(U^z)|_x=\varnothing$. But $U^z$ is a covering in the Zariski neighborhood $V=\operatorname{Im}(U^z\longrightarrow Y)$ of $y$ (the image of $U^z$ in $Y$ contains $y$ and $U^z$ is open in $U$ which has open image in $Y$ by \'etaleness). So $\varphi^{-1}(U^z)$ must be a covering in a Zariski neighborhood $\varphi^{-1}(V)$ of $|\varphi|^{-1}(y)\ni x$ because $\varphi^{-1}$ sends covers to covers. This is shown in the following diagram:
$$\begin{tikzcd}
                             & {|\varphi^{-1}(U^z)|} \arrow[ld, two heads] \arrow[d]  & {|U^z|} \arrow[rd, two heads] \arrow[d] &              \\
{x\in|\operatorname{Im}(\varphi^{-1}(V)\longrightarrow X)|} \arrow[rd] & {|\varphi^{-1}(U)|} \arrow[r, "|\varphi_U|"] \arrow[d] & {|U|} \arrow[d]                         & {|V|\ni y} \arrow[ld] \\
                             & {|X|} \arrow[r, "|\varphi|"]                           & {|Y|}                                   &             
\end{tikzcd}$$
So the fiber $|\varphi^{-1}(U^z)|_x$ is non-empty. Note that this part of the proof works if $K$ is any field and $x$ is any point.

    Injectivity: Pick $z\in U_y$ and define $U^z$ as before. We have constructed $U^z$ so that the fiber over $y$ of $U^z\longrightarrow Y$ is $z$.  We want to show that there is only one thing in $|\varphi^{-1}(U)|$ that maps to both $x$ and $z$. Consider the diagram
    $$\begin{tikzcd}
                                            & {|\varphi^{-1}(U^z)|} \arrow[ld, two heads] \arrow[d] \arrow[r] & {|U^z|}\ni z \arrow[d] \\
w\in{|\varphi_U|}^{-1}(|U^z|) \arrow[r, hook] & {|\varphi^{-1}(U)|} \arrow[r] \arrow[d]                         & {|U|}\ni z \arrow[d]   \\
                                            & x\in{|X|} \arrow[r]                                             & {|Y|}\ni y            
\end{tikzcd}$$
It suffices to show that there is a unique point $w$ in the fiber $\varphi^{-1}(U^z)_x$. Consider the projection map $\operatorname{pr}_1:U^z\times_Y U^z\longrightarrow U^z$. Since $\kbar$ is algebraically closed and $z$ is closed, $\operatorname{pr}_1^{-1}(z)= (z,z)$. So the diagonal $U^z\longrightarrow U^z\times_Y U^z$ is a Zariski open neighborhood of $\operatorname{pr}_1^{-1}(z)$ ($U^z$ is unramified over $Y$, so the diagonal is open). As in the proof of surjectivity, the diagonal $\varphi^{-1}(U)\longrightarrow\varphi^{-1}(U)\times_X \varphi^{-1}(U)$ is then a Zariski open neighborhood of the fiber over $w$ of $\operatorname{pr_1}:\varphi^{-1}(U)\times_X\varphi^{-1}(U)\longrightarrow \varphi^{-1}(U)$ because $\varphi^{-1}$ is functorial ($\varphi^{-1}$ sends diagonals to diagonals), preserves fiber products ($\varphi^{-1}(U\times_YU)=\varphi^{-1}(U)\times_{\varphi^{-1}(Y)}\varphi^{-1}(U)$ which is equal to $\varphi^{-1}(U)\times_{X}\varphi^{-1}(U)$ by Remark \ref{basedterminal}), and sends covers to covers. In particular, any pair $(w_1,w_2)\in\operatorname{pr}_1^{-1}(w)\subset \varphi^{-1}(U)\times_{X}\varphi^{-1}(U)$ lies in the diagonal so that there is a unique point in $|\varphi^{-1}(U^z)|_x$. This is summed up in the following diagram
\[\begin{tikzcd}
	& {|\varphi^{-1}(U^z)|} & {|U^z|} \\
	& {|\Delta(\varphi^{-1}(U^z))|} & {|\Delta(U^z)|} & {\ni(z,z)} \\
	{(w_1,w_2)\in} & {|\varphi^{-1}(U^z)\times_X \varphi^{-1}(U^z)|} & {|U^z\times_YU^z|} & {\ni (z,z)} \\
	{w\in} & {|\varphi^{-1}(U^z)|} & {|U^z|} & {\ni z} & {} \\
	{w\in} & {|\varphi^{-1}(U)|} & {|U|} & {\ni z} \\
	{x\in} & {|X|} & {|Y|} & {\ni y}
	\arrow[from=1-2, to=1-3]
	\arrow["\Delta"', two heads, from=1-2, to=2-2]
	\arrow["\Delta", two heads, from=1-3, to=2-3]
	\arrow[from=2-2, to=2-3]
	\arrow[hook, from=2-2, to=3-2]
	\arrow[hook, from=2-3, to=3-3]
	\arrow[maps to, from=2-4, to=3-4]
	\arrow[maps to, from=3-1, to=4-1]
	\arrow[from=3-2, to=3-3]
	\arrow["{\text{pr}_1}"', from=3-2, to=4-2]
	\arrow["{\operatorname{pr}_1}", from=3-3, to=4-3]
	\arrow[maps to, from=3-4, to=4-4]
	\arrow[maps to, from=4-1, to=5-1]
	\arrow[from=4-2, to=4-3]
	\arrow[from=4-2, to=5-2]
	\arrow[from=4-3, to=5-3]
	\arrow[maps to, from=4-4, to=5-4]
	\arrow[maps to, from=5-1, to=6-1]
	\arrow[from=5-2, to=5-3]
	\arrow[from=5-2, to=6-2]
	\arrow[from=5-3, to=6-3]
	\arrow[maps to, from=5-4, to=6-4]
	\arrow[shift right=3, curve={height=24pt}, maps to, from=6-1, to=6-4]
	\arrow[from=6-2, to=6-3]
\end{tikzcd}\] 
\end{proof}
\subsection{\'Etalified Rational Points}
Fix a perfect field $K$ of positive characteristic.
\begin{lem}[Rigidity of Geometric Points of \'Etale Sites, \cite{Voe}, \href{https://www.math.ias.edu/vladimir/sites/math.ias.edu.vladimir/files/Etale_topologies_published.pdf\#page=6}{Proposition 2.1}]\label{rigid} Let $X$ be a $\kbar$-scheme of finite type and pick two morphisms $(\varphi_1^{-1},\alpha_1),(\varphi_2^{-1},\alpha_2)\in\mor_{\kbar_\et}(\spec(\kbar)_\et,X_\et)$. If $\eta:\varphi_1^{-1}\overset{\sim}{\Rightarrow}\varphi_2^{-1}$ is a natural isomorphism, we get a $2$-isomorphism $(\varphi_1^{-1},\alpha_1)\overset{\sim}{\Rightarrow}(\varphi_2^{-1},\alpha_2)$. 
\end{lem}
\begin{proof}
    Take $U\in \kbar_\et$. It is a disjoint union of $\spec(\kbar)$'s. We would like to show that $$\begin{tikzcd}
                                                                               & (\varphi_1^{-1}\circ p_X^{-1})(U) \arrow[dd, "\:\eta(p_X^{-1}(U))", rightarrow] \\
U \arrow[ru, "\alpha_1(U)", rightarrow] \arrow[rd, "\alpha_2(U)"', rightarrow] &                                                                                      \\
                                                                               & (\varphi_2^{-1}\circ p_X^{-1})(U)                                                   
\end{tikzcd}$$ 
commutes. We get an automorphism $$\alpha_2(U)^{-1}\circ\eta(U)\circ\alpha_1(U):U\longrightarrow U$$ Since $\alpha_1,\alpha_2,$ and $\eta$ are natural isomorphisms, they commute with the action of $\operatorname{Aut}(U/\spec{\kbar})=S_{|\pi_0(U)|}$ (the symmetric group on $|\pi_0(U)|$ elements). So $\alpha_2(U)^{-1}\circ\eta(U)\circ\alpha_1(U)=\operatorname{id}_U$ if $\pi_0(U)\geq 3$ because $S_{\pi_0(U)}$ has trivial center when $\pi_0(U)\geq 3$. If $\pi_0(U)<3$, we may embed $U$ into an \'etale open $V$ with $\pi_0(V)\geq 3$ and deduce the same result by functoriality. Hence, $\alpha_2(U)=\eta(U)\circ\alpha_1(U)$.
\end{proof}
\begin{lem}[\cite{Voe}, \href{https://www.math.ias.edu/vladimir/sites/math.ias.edu.vladimir/files/Etale_topologies_published.pdf\#page=6}{Proposition 2.1}]
\label{2.1}
    Let $\overline{X}$ be a finite type $\kbar$-scheme. The natural map
    $$\xbar(\kbar)\longrightarrow\mor_{\kbar_\et}^\bullet(\spec(\kbar)_\et,\xbar_\et)$$
    is bijective.
\end{lem}
\begin{proof} Injectivity: Over $\kbar$, a $\kbar$-point is determined by its associated map on topological spaces. 

Surjectivity: Take $\varphi=(\varphi^{-1},\alpha)\in\mor^\bullet_{\kbar_\et}(\spec(\kbar)_\et,\xbar_\et)$. Consider the $\kbar$-point $\overline{x}$ corresponding to $\operatorname{Im}|\varphi|$ and let $\overline{x}_\et=(\overline{x}_\et^{-1},\alpha_x)$ be its induced map on \'etale sites over $\spec(\kbar)_\et$. We have $\overline{x}_\et^{-1}(U)=U_{\overline{x}}$ for all \'etale $U\longrightarrow X$. By Lemma \ref{admis}, $|\varphi|$ gives us an isomorphism $$|\varphi^{-1}(U)_{\spec(\kbar)}|=|\varphi^{-1}(U)|\overset{\sim}{\longrightarrow}|U_{\overline{x}}|=|\overline{x}_\et^{-1}(U)|$$  Since $\kbar$ is algebraically closed and $\varphi^{-1}(U)$ and are $U_{\overline{x}}$ reduced, the bijection on underlying topological spaces promotes to an isomorphism of schemes $\varphi^{-1}(U)\overset{\sim}{\longrightarrow} U_{\overline{x}}$. Hence, we get a natural isomorphism $\varphi^{-1}\overset{\sim}{\Rightarrow}\overline{x}_\et^{-1}$. Lemma \ref{rigid} then tells us that we have a $2$-isomorphism $\varphi\overset{\sim}{\Rightarrow}\overline{x}_{\et}$. So we have found a geometric point that maps to $\varphi$.
\end{proof}
\begin{prop}[\cite{Voe}, \href{https://www.math.ias.edu/vladimir/sites/math.ias.edu.vladimir/files/Etale_topologies_published.pdf\#page=6}{Proposition 2.2}]\label{injprep}
    Let $X$ be a finite type $K$-scheme. The map
    $$\beta:\xbar(\kbar)\longrightarrow\mor_{K_\et}^\bullet(\spec(\kbar)_\et,X_\et)$$
    $$\overline{x}\mapsto (\operatorname{pr}_X)_\et\circ \overline{x}_\et$$
    is bijective \emph($\operatorname{pr}_X:\xbar\longrightarrow X$ is the projection\emph).
\end{prop}
\begin{proof} 
 Consider $\varphi\in\mor_{K_\et}^\bullet(\spec(\kbar)_\et,X_\et)$. By Proposition \ref{base}, we get a morphism $\varphi_{\kbar}$ fitting into the diagram
    $$\begin{tikzcd}
(\spec(\kbar)\times_K\spec(\kbar))_\et \arrow[r, "\varphi_{\kbar}"] \arrow[d] & \xbar_\et \arrow[d, "(\operatorname{pr}_X)_\et"] \\
\spec(\kbar)_\et \arrow[r, "\varphi"]                                       & X_\et              
\end{tikzcd}$$
Let $\Delta:\spec(\kbar)\longrightarrow\spec(\kbar)\times_K\spec(\kbar)$ be the diagonal and set $\alpha(\varphi)=\operatorname{Im}|\varphi_{\kbar}\circ\Delta_\et|$ (Lemma \ref{2.1}). We'll show that $\alpha=\beta^{-1}$. We have \begin{align*}
    \beta(\alpha(\varphi))&=\beta(\operatorname{Im}|\varphi_{\kbar}\circ\Delta_\et|)\\
    &=(\operatorname{pr}_X)_\et\circ\varphi_{\kbar}\circ\Delta_\et\\
    &=\varphi
\end{align*} 
by the commutativity of the diagram (note that we are conflating $\operatorname{Im}|\varphi_{\kbar}\circ\Delta_\et|$ with the geometric point corresponding to it). In the other direction we have \begin{align*}
    \alpha(\beta(\overline{x}))&=\alpha((\operatorname{pr}_X)_\et\circ\overline{x}_\et)\\
    &=\operatorname{Im}|((\operatorname{pr}_X)_\et\circ\overline{x}_\et)_{\kbar}\circ\Delta_\et|\\
    &=\operatorname{Im}|\overline{x}_\et|\\&=\overline{x}
\end{align*}
The third equality is easier to see while contemplating the following diagram:
$$\begin{tikzcd}
(\spec(\overline{K})\times_K\spec(\overline{K}))_{\et} \arrow[d] \arrow[rd, "((\operatorname{pr}_X)_{\et}\circ\overline{x}_{\et})_{\kbar}"] &                                                             &         \\
\spec(\kbar)_{\et} \arrow[r, "\overline{x}_{\et}"] \arrow[u, "\Delta_{\et}", bend left]                                                     & \overline{X}_{\et} \arrow[r, "(\operatorname{pr}_X)_{\et}"] & X_{\et}
\end{tikzcd}$$
\end{proof}
\begin{notn}
    Let $X$ be the perfection of a finite type $K$-scheme. Given any $\varphi\in\mor_{K_\et}^\bullet(X_\et,Y_\et)$, we let $\overline{\varphi}\in\mor(X(\kbar),Y(\kbar))$ denote the corresponding map on geometric points coming from Proposition \ref{injprep}. More explicitly we have 
    $$\begin{tikzcd}
X(\kbar) \arrow[r, "\sim"] \arrow[d, "\overline{\varphi}"] & {\mor_{K_\et}^\bullet(\spec(\kbar)_\et,X_\et)} \arrow[d, "\varphi_\et"] \\
Y(\kbar) \arrow[r, "\sim"]                      & {\mor_{K_\et}^\bullet(\spec(\kbar)_\et,Y_\et)}
\end{tikzcd}$$
    Observe that, if $f$ is a morphism of schemes and $f^*$ is the map it induces on geometric points, we have $\overline{f_\et}=f^*$.
\end{notn}
We can now prove injectivity.
\begin{prop}\label{inj}
    Let $X=X_0^\perf$ be the perfection of a $K$-scheme of finite type. Let $Y$ be a finite type $K$-scheme. The natural map $$\mor_{K}(X,Y)\longrightarrow\mor_{K_\et}^\bullet(X_\et,Y_\et)$$
    is injective.
\end{prop}
\begin{proof}
Consider the composite
$$\mor_{K}(X,Y)\longrightarrow\mor(X(\kbar),Y(\kbar))\longrightarrow \mor(\mor_{K_\et}^\bullet(\spec{\kbar}_\et,X_\et),\mor_{K_\et}^\bullet(\spec{\kbar}_\et,Y_\et))$$
Proposition \ref{injprep} gives us
$$\mor(\mor_{K_\et}^\bullet(\spec{\kbar}_\et,X_\et),\mor_{K_\et}^\bullet(\spec{\kbar}_\et,Y_\et))\cong \mor(\xbar(\kbar),\overline{Y}(\kbar))$$
So we would like to show that the composite
$$\mor_{K}(X,Y)\longrightarrow\mor(X(\kbar),Y(\kbar))\longrightarrow\mor(\xbar(\kbar),\overline{Y}(\kbar))$$
is injective. First, notice that $X\cong (X_{0, \operatorname{red}})^\perf=X_{0,\operatorname{red}}^\perf$. For $X_{0,\operatorname{red}}$ and $Y$ of finite type over $K$, a $K$-morphism is determined by what it does on geometric points. This remains true if we replace $X_{0,\operatorname{red}}$ with its perfection because $X_{0,\operatorname{red}}^\perf$ is reduced and $|X_{0,\operatorname{red}}|\cong|X_{0,\operatorname{red}}^\perf|$ (in particular, being of finite type is relevant only for the density of closed points which still holds for the perfection).

Take $f\in\mor_{K}(X,Y)$. By the universal properties of $\xbar$ and $\overline{Y}$, every $f\in\mor_{K}(X,Y)$ extends to a unique $\overline{f}\in\mor_{\kbar}(\xbar,\overline{Y})$ such that, for every $\overline{x}\in X(\kbar)$, the following diagram commutes: $$
    \begin{tikzcd}
                                                                             & \overline{X} \arrow[r, "\overline{f}"] \arrow[d] & \overline{Y} \arrow[d] \\
\spec(\kbar) \arrow[ru, "\overline{\overline{x}}"] \arrow[r, "\overline{x}"] & X \arrow[r, "f"]                                 & Y                     
\end{tikzcd}
$$ The image of $f$ in $\mor(\xbar(\kbar),\overline{Y}(\kbar))$ is the map that sends a geometric point $\overline{x}\in\xbar(\kbar)$ to $\overline{f}\circ\overline{x}$. Take $f,g\in\mor_{K}(X,Y)$ such that $f\circ\overline{x}\neq g\circ\overline{x}$ for some $\overline{x}\in X(\kbar)$. We want to show that $\overline{f}\circ\overline{\overline{x}}\neq\overline{g}\circ\overline{\overline{x}}$. Assume otherwise. Then, by the commutativity of the diagram, $f\circ\overline{x}=g\circ\overline{x}$---a contradiction.
\end{proof}
\begin{rmk}
    Observe that injectivity also holds if $X$ is a reduced $K$-scheme of finite type and/or $Y$ is the perfection of a finite type $K$-scheme.
\end{rmk}
Before proving the main technical proposition, we will record a small lemma for ease of presentation.
\begin{lem}\label{etalemono}
    Let $k$ be a separably closed field. An unramified morphism of $k$-schemes that is injective on underlying toplogical spaces is a monomorphism.
\end{lem}
\begin{proof}
    Let $f:X\longrightarrow Y$ be an unramified injection of $k$-schemes. Since $k$ is separably closed, $f$ is radicial (injective with purely inseparable residue field extensions). So by [\cite{Stacks}, \href{https://stacks.math.columbia.edu/tag/01S4}{Tag 01S4}], the diagonal is surjective. Hence the diagonal is a surjective open immersion and thus an isomorphism.
\end{proof}
The following proposition is key for determining if a morphism of \'etale sites comes from a morphism of schemes. The idea is that a morphism of sites comes from a morphism of schemes if its induced map on geometric points \'etale locally coincides with the induced map on geometric points of some morphism of schemes.
\begin{prop}[\cite{Voe}, \href{https://www.math.ias.edu/vladimir/sites/math.ias.edu.vladimir/files/Etale_topologies_published.pdf\#page=7}{Proposition 2.3}]\label{2.4}
    Let $X$ be a reduced, finite type $K$-scheme or the perfection of a finite type $K$-scheme. Let $Y$ be a finite type $K$-scheme. Take $\varphi\in\mor_{K_\et}^\bullet(X_\et,Y_\et)$. Assume that for all \'etale $U\longrightarrow Y$, there exists some $\widetilde{\varphi}_U\in\mor_{K}(\varphi^{-1}(U),U)$ such that $\widetilde{\varphi}_U(\overline{x})=\overline{\varphi_U}(\overline{x})$ for all $\overline{x}\in\varphi^{-1}(U)(\kbar)$. Then $\varphi$ comes from a morphism of schemes.
\end{prop}
\begin{proof}
    We will prove the theorem for a reduced, finite type $K$-scheme $X$. As in Proposition \ref{inj}, we may replace $X$ with its perfection and the proof goes through just the same. We would like to show that the \'etalification of $\widetilde{\varphi}_Y=\widetilde{\varphi}$ is $\varphi=(\varphi^{-1},\alpha)$, that is, we want to show that there is a 2-isomorphism $\widetilde{\varphi}_\et=(\widetilde{\varphi}_\et^{-1},\widetilde{\alpha})\cong(\varphi^{-1},\alpha)=\varphi$. Consider the diagram
    $$
\begin{tikzcd}
\varphi^{-1}(U) \arrow[rrd, "\widetilde{\varphi}_U", shift left=3] \arrow[rdd] &                                                                        &             \\
                                                                               & \widetilde{\varphi}^{-1}_\et(U) \arrow[r] \arrow[d] & U \arrow[d] \\
                                                                               & X \arrow[r, "\widetilde{\varphi}"]                                     & Y          
\end{tikzcd}$$
The square commutes by construction because $\widetilde{\varphi}_\et^{-1}(U)=X\times_{\widetilde{\varphi}}U$. By assumption, the whole diagram commutes on geometric points ($\widetilde{\varphi}$ and $\widetilde{\varphi}_U$ are the same as $\varphi$ and $\varphi|_U$ on geometric points). Since $X$ is reduced, $\varphi^{-1}(U)$ is reduced because it is \'etale over $X$. Everything in sight is of finite type over $K$ because we are using the quasicompact (and separated) \'etale site. Hence, the diagram commutes. So we get a diagram
$$\begin{tikzcd}
\varphi^{-1}(U) \arrow[rrd, "\widetilde{\varphi}_U", shift left=3] \arrow[rdd] \arrow[rd, "\exists!", dotted] &                                                                        &             \\
                                                                                                              & \widetilde{\varphi}^{-1}_\et(U) \arrow[r] \arrow[d] & U \arrow[d] \\
                                                                                                              & X \arrow[r, "\widetilde{\varphi}"]                                     & Y          
\end{tikzcd}$$
We may compatibly do this for all $U$, so we get a natural transformation $\eta:\varphi^{-1}\Rightarrow \widetilde{\varphi}^{-1}_\et$. We would like to show that $\eta$ is a natural isomorphism. Lift the diagram to $\kbar$. It suffices to show that $\eta_{\kbar}$ is an isomorphism by fpqc descent. To show that $\eta_{\kbar}(\overline{U}):\varphi^{-1}(\overline{U})\longrightarrow\widetilde{\varphi}_\et^{-1}(\overline{U})$ is an isomorphism, it suffices to show that it induces a bijection $|(\varphi^{-1}(\overline{U}))_{\overline{x}}|\longrightarrow |(\widetilde{\varphi}^{-1}_\et(\overline{U}))_{\overline{x}}|$ for all $\overline{x}\in\xbar(\kbar)$. Indeed, $\varphi^{-1}(\overline{U})\longrightarrow\widetilde{\varphi}^{-1}_\et(\overline{U})$ is \'etale by cancellation, so if $\eta_{\kbar}(\overline{U})$ induces a bijection on fibers over closed points of $\xbar$, it is surjective (because surjectivity is a pointwise condition) and monic by Lemma \ref{etalemono} (note that $K$ is perfect). Thus $\eta_{\kbar}(U)$ would be a surjective \'etale monomorphism, and \'etale monomorphisms are open immersions by [\cite{Stacks}, \href{https://stacks.math.columbia.edu/tag/025G}{Tag 025G}].

Lemma $\ref{admis}$ gives us isomorphisms
\begin{align*}
|\varphi^{-1}(\overline{U})_{\overline{x}}|&\overset{\sim}{\longrightarrow}|\overline{U}_{{|\varphi|}(\overline{x})}|\\
|\widetilde{\varphi}^{-1}_\et(\overline{U})_{\overline{x}}|&\overset{\sim}{\longrightarrow}|\overline{U}_{|{\widetilde\varphi}_\et|(\overline{x})}|=|\overline{U}_{\widetilde{\varphi}(\overline{x})}|    
\end{align*}
for all $\overline{x}\in \xbar(\kbar)$. By assumption, $|\varphi|(\overline{x})=\widetilde{\varphi}(\overline{x})$ (where we conflate $\widetilde{\varphi}$ with the map it induces on geometric points). So we have a commutative diagram 
\[\begin{tikzcd}
	{|\varphi^{-1}(\overline{U})_{\overline{x}}|} & {|\overline{U}_{|\varphi|(\overline{x})}|} \\
	{|\widetilde{\varphi}^{-1}_\et(\overline{U})_{\overline{x}}|} & {|\overline{U}_{\widetilde{\varphi}(\overline{x})}|}
	\arrow["\sim", from=1-1, to=1-2]
	\arrow["{\eta_{\overline{K},x}}"', from=1-1, to=2-1]
	\arrow[no head, from=1-2, to=2-2]
	\arrow[shift left, no head, from=1-2, to=2-2]
	\arrow["\sim", from=2-1, to=2-2]
\end{tikzcd}\]
We conclude that $\eta_{\kbar,x}$ is an isomorphism for all $\overline{x}\in \xbar(\kbar)$ as desired.

Now we need to show that $\eta$ is a 2-isomorphism over $\spec(K)_\et$. It suffices to show that, for every finite extension $E\supset K$, the following diagram (of isomorphisms) commutes:
$$\begin{tikzcd}
                                                                          & \varphi^{-1}(Y_E) \arrow[rd, "\eta(Y_E)"] &                                   \\
X_E \arrow[ru, "\alpha(E)"] \arrow[rr, "{\widetilde{\alpha}}(E)"] &                                            & \widetilde{\varphi}^{-1}_\et(Y_E)
\end{tikzcd}$$
Indeed, we need to check commutativity after pulling back along each \'etale open of the base (recall that a 2-morphism involves composing with $p_X^{-1}$ and $p_Y^{-1}$). This is a diagram of isomorphisms over $X$, so we can check commutativity on the fibers over closed points of $X$ (using the fact that $X$ is reduced and of finite type). Let $\overline{x}$ be a geometric point of $X$. We have a natural isomorphism $\eta:\varphi^{-1} \overset{\sim}{\longrightarrow}\widetilde{\varphi}^{-1}_\et$. So we get a natural isomorphism $$\eta_{\overline{x}}:\overline{x}_\et^{-1}\circ\varphi^{-1}\overset{\sim}{\longrightarrow}\overline{x}^{-1}_\et\circ\widetilde{\varphi}^{-1}_\et$$
We want to show that the following diagram commutes
\[\begin{tikzcd}
	& {x_\et^{-1}(\varphi^{-1}(Y_E))=\varphi^{-1}(Y_E)_x} \\
	{\spec(\kbar)\times_K\spec(E)} && {x_\et^{-1}(\widetilde\varphi_\et^{-1}(Y_E))=\widetilde\varphi^{-1}_\et(Y_E)_x}
	\arrow["{\eta_x(Y_E)}", from=1-2, to=2-3]
	\arrow["{\alpha_x(E)}", from=2-1, to=1-2]
	\arrow["{\widetilde{\alpha}_x(E)}", from=2-1, to=2-3]
\end{tikzcd}\]
Everything is reduced and of finite type so we can check commutativity on geometric points. Geometric points depend only on the base change of everything to $\kbar$. And the diagram commutes after lifting everything to $\kbar$ by Lemma \ref{rigid} (applied to $\eta_{x,\kbar}$).
\end{proof}
\begin{prop}\label{fpqc}
    Let $X$ and $Y$ be $K$-schemes of finite type with $X$ reduced. Let $\varphi\in\mor_{K}^\bullet(X_\et,Y_\et)$ and let $E$ be an algebraic extension of $K$. If there exists some $\widetilde{\varphi}_E\in\mor_{E}(X_E,Y_E)$ such that $\widetilde{\varphi}_E(\overline{x})=\overline{\varphi_E}(\overline{x})$ for all $\overline{x}\in X_E(\kbar)$, then there exists some $\widetilde{\varphi}\in\mor_{K}(X,Y)$ such that $\widetilde{\varphi}(\overline{x})=\overline{\varphi}(\overline{x})$ for all $\overline{x}\in X(\kbar)$.
\end{prop}
\begin{proof}
    By fpqc descent we have an equalizer diagram $$
        \begin{tikzcd}
{\mor_{K}(X,Y)} \arrow[r] & {\mor_E(X_E,Y_E)} \arrow[r, shift right] \arrow[r, shift left] & {\mor_{K}(X_E\times_XX_E,Y)}
\end{tikzcd}
   $$
   The equalizer condition is determined by what happens on geometric points because $X$ and $Y$ are of finite type and $X$ is reduced. So the equalizer condition for $\widetilde{\varphi}_E\in\mor_E(X_E,Y_E)$ says that $\widetilde{\varphi}_E(\overline{x})=\overline{\varphi_E}(\overline{x})$ for all $x\in X_E(\overline{K})$. The equalizer is $\mor_{K}(X,Y)$ so there exists a $\widetilde{\varphi}\in\mor_{K}(X,Y)$ with $\widetilde{\varphi}(\overline{x})=\overline{\varphi}(\overline{x})$ for all $\overline{x}\in X(\kbar)$.\end{proof}
\begin{rmk}
    As before, Proposition \ref{fpqc} is also true if $X$ is the perfection of a finite type $K$-scheme.
\end{rmk}
\begin{lem}[\cite{Voe}, \href{https://www.math.ias.edu/vladimir/sites/math.ias.edu.vladimir/files/Etale_topologies_published.pdf\#page=8}{Proposition 2.5}]\label{2.6l}
    Let $X$ be a reduced, finite type $K$-scheme. Let $E$ be an algebraic extension of $K$. The natural map $$\mor_{K}(\spec(E),X)\longrightarrow\mor_{K_\et}^\bullet (\spec(E)_\et,X_\et)$$
    is a bijection.
\end{lem}
\begin{proof}
    Proposition \ref{inj} gives us injectivity. Take $\varphi\in\mor_{K_\et}(\spec(E)_\et,X_\et)$. By Proposition \ref{base}, we get $\varphi_{\kbar}$ and a commutative diagram
    $$\begin{tikzcd}
(\spec(E)\times_K\spec(\kbar))_\et \arrow[r, "\varphi_{\kbar}"] \arrow[d] & \xbar_\et \arrow[d] \\
\spec(E)_\et \arrow[r, "\varphi"]                                         & X_\et              
\end{tikzcd}$$
Now $\spec(E)\times_K\spec(\kbar)$ is just a disjoint union of $\spec(\kbar)$'s. So by Lemma \ref{2.1}, $\varphi_{\kbar}$ comes from a morphism of schemes. Proposition \ref{fpqc} then tells us that $\varphi$ coincides with some morphism of schemes on geometric points. We may conclude the same for $\varphi_U$ for all $U\in X_\et$. So $\varphi$ comes from a morphism of schemes by Proposition \ref{2.4}. 
\end{proof}
This gives us [\cite{Voe}, \href{https://www.math.ias.edu/vladimir/sites/math.ias.edu.vladimir/files/Etale_topologies_published.pdf#page=9}{Proposition 2.8}]: 
\begin{prop}\label{perfsq}
    Let $X=X_0^\perf$ be the perfection of a geometrically connected $K$-scheme of finite type. Assume that $X(K)\neq\varnothing$. Let $Y$ be a geometrically connected $K$-scheme of finite type. Then for any $\varphi\in\mor^\bullet_{K_\et}(X_\et,Y_\et)$, we have a commutative diagram
    $$
        \begin{tikzcd}
X(\kbar) \arrow[d] \arrow[rr, "\overline{\varphi}"]                               &  & Y(\kbar) \arrow[d]                                          \\
{\underset{E\supset K}{\varinjlim}H^1\big(\Gamma_E,\pixbar\big)} \arrow[rr] &  & {\underset{E\supset K}{\varinjlim}H^1\big(\Gamma_E,\pi_1^{\et}(\overline{Y},\overline{\varphi}(\overline{x}))\big)}
\end{tikzcd}
    $$
    where the colimit is indexed over finite extensions of $K$ and $H^1$ denotes group cohomology.
\end{prop}
\begin{proof}
Observe that we have $X(E)\cong X_{0,\operatorname{red}}(E)$ for all $E\supset K$ because $E$ is perfect and $X\cong X_{0,\operatorname{red}}^\perf$. Let $\varphi=(\varphi^{-1},\alpha):X_\et\longrightarrow Y_\et$ be an admissible morphism. Then we get a diagram $$\begin{tikzcd}
X(E) \arrow[d] \arrow[r]        & Y(E) \arrow[d]                                                                   \\
{H^1\big(\Gamma_E,\pixbar\big)} & {H^1\big(\Gamma_E,\pi_1^\et(\overline{Y},\overline{\varphi}(\overline{x}))\big)}
\end{tikzcd}$$ where the top map, which is induced by $\varphi$, comes from Lemma \ref{2.6l}, and the vertical maps come from sections of the homotopy exact sequence (see e.g., [\cite{Stix}, \href{https://link.springer.com/book/10.1007/978-3-642-30674-7}{Proposition 8/Definition 20}]). We know that $\overline{\varphi}$ sends geometric points to geometric points, so we get a map $\pi(\varphi):\pi_1^\et(X,\overline{x})\longrightarrow\pi_1^\et({Y},\overline{\varphi}(\overline{x}))$. Note that $\pi_1^\et$ is unaffected by perfection. Since $\varphi^{-1}$ sends coverings to coverings, we get a morphism of homotopy exact sequences 
\[\begin{tikzcd}
	& {\pi_1^\et(\xbar,\overline{x})} & {\pi_1^\et(X,\overline{x})} & {\Gamma_K} \\
	1 &&&& 1 \\
	& {\pi_1^\et(\overline{Y},\overline{\varphi}(\overline{x}))} & {\pi_1^\et({Y},\overline{\varphi}(\overline{x}))} & {\Gamma_K}
	\arrow[from=1-2, to=1-3]
	\arrow["{\pi(\varphi_{\kbar})}", from=1-2, to=3-2]
	\arrow[from=1-3, to=1-4]
	\arrow["{\pi(\varphi)}", from=1-3, to=3-3]
	\arrow[from=1-4, to=2-5]
	\arrow[from=1-4, to=3-4]
	\arrow[from=2-1, to=1-2]
	\arrow[from=2-1, to=3-2]
	\arrow[from=3-2, to=3-3]
	\arrow[from=3-3, to=3-4]
	\arrow[from=3-4, to=2-5]
\end{tikzcd}\]

We've assumed that $X(K)\neq\varnothing$, so the exact sequence splits and $\pi(\varphi_{\kbar})$ is $\Gamma_K$-equivariant. In particular, we get a map $H^1\big(\Gamma_E,\pixbar\big)\longrightarrow H^1\big(\Gamma_E,\pi_1^\et(\overline{Y},\overline{\varphi}(\overline{x}))\big)$ for every finite extension $E\supset K$. Then we take a colimit over extensions of $K$ and get a map $$\underset{E\supset K}{\varinjlim}\:H^1\big(\Gamma_E,\pixbar\big)\longrightarrow\underset{E\supset K}{\varinjlim}\:H^1\big(\Gamma_E,\pi_1^\et(\overline{Y},\overline{\varphi}(\overline{x}))\big)$$ making the desired diagram commute. 
\end{proof}
\section{The Meat}
\subsection{Kummer Theory Preliminaries}
Fix a perfect field $K$ of positive characteristic. To reconstruct a morphism $X_\et\rightarrow Y_\et$, it turns out that we can replace $Y$ with $\GG_m$. This is nice because it allows us to use Kummer theory. The proof of [\cite{Voe}, \href{https://www.math.ias.edu/vladimir/sites/math.ias.edu.vladimir/files/Etale_topologies_published.pdf\#page=11}{Proposition 3.1}] goes through without modification.
\begin{prop}\label{suff}
    Let $X_0$ and $Y$ be finite type $K$-schemes. Let $X=X_0^\perf$. Assume that for all geometrically connected and affine $U$ that are \'etale over $X$ such that $U(E)\neq\varnothing$ for some finite extension $E\supset K$ and all $\varphi\in\mor_{E_\et}^\bullet(U_{E,\et},(\GG_{m,E})_\et)$, there exists a morphism $\widetilde{\varphi}\in\mor_E(U_E,\GG_{m,E})$ such that $\widetilde{\varphi}(\overline{x})=\overline{\varphi}(\overline{x})$ for all $\overline{x}\in U_E(\kbar)$. Then $$\mor_K(X,Y)\longrightarrow\mor_{K_\et}^\bullet(X_\et,Y_\et)$$ is surjective (and hence bijective).
\end{prop}
\begin{proof}
    By Proposition \ref{2.4}, it suffices to show that, for any admissible $\varphi:X_\et\longrightarrow Y_\et$ and any $U\in Y_\et$, the map $\varphi_U:\varphi^{-1}(U)\longrightarrow U$ coincides with some morphism of schemes on geometric points of $\varphi^{-1}(U)$. For the following argument, we can just consider $X$. Proposition \ref{fpqc} tells us that coincidence over $E$ is enough, and we may pass to finite $E\supset K$ such that $X(E)\neq\varnothing$. For notational simplicity, we will assume that $K=E$. We may further reduce to the case where $X=\spec(A)$ and $Y=\spec(B)$. And we may also assume that $X$ is geometrically connected. 

    Take $\varphi\in\mor_{K_\et}^\bullet(X_\et,Y_\et)$. Consider $b\in B=\mor_K(Y,\AAa^1_K)$ and abusively conflate it with the map it induces on geometric points. We have a map $\overline{b_\et\circ\varphi}=b\circ\overline{\varphi}:X(\kbar)\longrightarrow \AAa^1_K(\kbar)$. Cover $\AAa^1_K$ with the two copies of $\GG_m$ given by $\AAa^1_K\minus\{0\}$ and $\AAa^1_K\minus \{1\}$. For $i=0,1$, let $X_i$ be the open subscheme of $X$ corresponding to the open set $|X|\minus|b_\et\circ \varphi|^{-1}(i)$ (recall that $|b_\et|=|b|$). By assumption, there exist morphisms $\psi_{0}(b):X_0\longrightarrow \GG_{m,K}$ and $\psi_{1}(b):X_1\longrightarrow \GG_{m,K}$ whose induced maps on geometric points agree with $b\circ\overline{\varphi}|_{X_{0}(\kbar)}$ and $b\circ\overline{\varphi}|_{X_{1}(\kbar)}$ respectively. These morphisms of schemes agree on $X_0\cap X_1$ because their induced maps on geometric points agree on $X_0(\kbar)\cap X_1(\kbar)$ (using the fact that $X$ is reduced and $\GG_{m,K}$ is of finite type over $K$). Since $X=X_0\cup X_1$, $\psi_0(b)$ and $\psi_1(b)$ glue along $X_0\cap X_1$ to a morphism $\psi(b):X\longrightarrow\AAa^1_K$ that agrees with $b_\et\circ\varphi$ on geometric points. We thus get a map of sets \begin{align*}
        \psi:B&\longrightarrow \mor_K(X,\AAa^1_K)=A\\
        b&\longmapsto \psi(b)
    \end{align*}
    We would like to show that $\psi$ is actually a ring homomorphism. On geometric points, $\psi(b+b')$ is equal to $(b+b')\circ \overline{\varphi}$. But this is the same as $b\circ\overline{\varphi}+b'\circ\overline{\varphi}\in\mor(X(\kbar),\AAa^1_K(\kbar))$ where the ring structure on $\mor(X(\kbar),\AAa^1_K(\kbar))$ is induced by the ring structure on $\AAa^1_K(\kbar)=\kbar$. By the uniqueness of $\psi$, this implies that $\psi({b+b'})=\psi(b)+\psi({b'})$. The same argument works for $bb'$. So $\spec(\psi)$ is a map of schemes. We know that $\psi(b)$ coincides with $b_\et\circ\varphi$ on geometric points for all $b\in B=\mor_K(Y,\AAa_K^1)$. So $\spec(\psi)$ coincides with $\varphi$ on $X(\kbar)$ (noting that $Y$ is determined by its global sections).
\end{proof}
From now on, we let $X_0$ be a geometrically connected, affine scheme of finite type over $K$ with $X_0(K)\neq\varnothing$ and let $X=X_0^\perf$. This is all we need to consider because of the above proposition. The following lemma is well known.
\begin{lem}\label{perf}
    If $f:U\longrightarrow X$ is \'etale and $X$ is perfect, $U$ is perfect.
\end{lem}
\begin{proof}
    Let $\Phi_{(-)}$ denote the absolute Frobenius on $(-)$. Let $U^{(p/X)}=U\times_{\Phi_X}X$ and let $\Phi_{U/X}:U\longrightarrow U^{(p/X)}$ denote the relative Frobenius of $U$ over $X$ [\cite{Stacks}, \href{https://stacks.math.columbia.edu/tag/0CC9}{Tag 0CC9}]. We see that $\operatorname{pr}_U:U^{(p/X)}\longrightarrow U$ is an isomorphism because $\Phi_X$ is. The composition $U\overset{\Phi_{U/X}}{\longrightarrow} U^{(p/X)}\overset{\operatorname{pr}_X}{\longrightarrow} X$ is \'etale because it is equal to $f$. And $\operatorname{pr}_X$ is \'etale. So $\Phi_{U/X}$ is \'etale by cancellation. By [\cite{Stacks}, \href{https://stacks.math.columbia.edu/tag/0CCB}{Tag 0CCB}], $\Phi_{U/X}$ is a universal homeomorphism. It is also \'etale, so it is an isomorphism by [\cite{Stacks}, \href{https://stacks.math.columbia.edu/tag/025G}{Tag 025G}]. We conclude that $\Phi_U=\operatorname{pr}_U\circ\:\Phi_{U/X}$ is an isomorphism.
\end{proof}
We have the Kummer short exact sequence on the \'etale site of $\xbar$ 
$$1\longrightarrow \mu_{n,\xbar}\longrightarrow \GG_{m,\xbar}\overset{t\,\mapsto t^n}{\longrightarrow}\GG_{m,\xbar}\longrightarrow 1$$
for all $n$. In general, we only get it for $n$ prime-to$-p$, but Frobenius is an isomorphism, so by Lemma \ref{perf}, every section is \'etale locally a $p^m$-th root for all $m$. The long exact sequence in \'etale cohomology gives us
$$1\longrightarrow \OO^*(\xbar)/\OO^*(\xbar)^n\overset{\chi_n}{\longrightarrow} H^1_\et(\xbar,\mu_n)\longrightarrow\pic(\overline{X})[n]\longrightarrow 1$$
where $\pic(\overline{X})[n]$ denotes $n$-torsion in the Picard group. Taking the limit over $n$, we have $$1\longrightarrow\ohat\overset{\chi}{\longrightarrow}\hone\longrightarrow \widehat{T}(\pic(\xbar))$$
Note that $\OO^*(\xbar)/\OO^*(\xbar)^{p^m}=0$ for all $m$. So $\ohat=\ohat_\ptp$ (where $\ptp$ denotes taking the prime-to$-p$ limit) and $H^1_\et(\xbar,\mu_{p^m})\cong\pic(\overline{X})[p^m]$ for all $m$. The following lemma is well known (see e.g., [\cite{BS}, \href{https://arxiv.org/pdf/1507.06490\#page=11}{Lemma 3.5}]):
\begin{lem}\label{picperf}
    Let $X$ be a qcqs $\FF_p$-scheme. Pullback along $X^\perf\rightarrow X$ induces an isomorphism $\pic(X^\perf)\cong\pic(X)[1/p]$.
\end{lem}
\begin{proof}
    $X^\perf= \lim_\Phi X$ so $\pic(X^\perf)\cong \colim\pic(X)$ where the colimit is taken over pullbacks of Frobenius which raise each line bundle to its $p$-th power.
\end{proof} So $\pic(\overline{X})[p^m]=H^1_\et(\xbar,\mu_{p^m})=0$ for all $m$ because $p$ is invertible in $\pic(\xbar)=\pic(\xbar^\perf)$.
Therefore $\hone$ is the same as its prime-to$-p$ part.

Let $\otop(\overline{X})=\mor_{\overline{K}_\et}(\overline{X}_{\et},(\GG_{m,\overline{K}})_\et)$. Fix a geometric point $\overline{x}$ of $X$. By the functoriality of $\pi_1^\et$, we get a map $$\otop(\overline{X})\longrightarrow\Hom_\cts(\pi_1^\et(\overline{X},\overline{x}), \pigmbar)$$ In the case where $K$ is an AFG field of characteristic $0$ and $X$ is a finite type $K$-scheme, we have $\pigmbar\cong\widehat\ZZ$, so we get an isomorphism $$\Hom_\cts(\pixbar,\pigmbar)\overset{\sim}{\longrightarrow}\hone$$ and the composite
$$\otop(X)\longrightarrow\otop(\overline{X})\longrightarrow\Hom_\cts(\pixbar,\pigmbar)\overset{\sim}{\longrightarrow}\hone$$
is our $\chi_\top$ where the map $\otop(X)\longrightarrow\otop(\xbar)$ comes from Proposition \ref{base}. The image of $\chi_\top$ lies in $\hone^{\Gamma_K}$ (where the action of $\Gamma_K$ comes from the splitting of the homotopy exact sequence).

In characteristic $p$, $\pigmbar$ is a bit unwieldy because of Artin-Schreier covers. So if we want an analogue of $\chi_\top$, we should pass to the tame fundamental group $\pigmtame\cong\widehat\ZZ_\ptp\cong\prod_{\ell\neq p}\ZZ_\ell$. 
The tame fundamental group is a quotient of $\pigmbar$ which sits in an exact sequence
$$1\longrightarrow \pigmbar^{w}\longrightarrow\pigmbar\longrightarrow\pigmtame\longrightarrow1$$
where $\pigmbar^{w}$---the wild fundamental group---is pro$-p$. In this case, the tame fundamental group is the same as the prime-to$-p$ fundamental group. The exact sequence splits by a profinite version of the Schur-Zassenhaus Theorem [\cite{RZ}, \href{https://link.springer.com/book/10.1007/978-3-662-04097-3}{Theorem 2.3.15}]. The tame fundamental group is the ``Kummer part" of $\pigmbar$ ($t\mapsto t^n$, $p\nmid n$) while $\pigmbar^{w}$ keeps track of Artin-Schreier covers ($t\mapsto t^n-t-f(t)$, $p\nmid\operatorname{deg}(f)$). We want to isolate the Kummer part of $\pigmbar$ because the image of the boundary map $\chi:\ohat\longrightarrow\hone$ comprises prime-to$-p$ projective limits of Kummer classes. Recall that the Kummer classes in $H^1_\et(\xbar,\mu_n)$ are the classes associated to the \'etale $\mu_n$-torsors that are given by extracting an $n$-th root of a unit; so if $\xbar=\spec(A)$, such torsors are given by $\spec(A[T]/T^n-u)$, $u\in\OO^*(\xbar)$.

Given this, we will define $\chi_\top$ as the composite
$$\otop(X)\longrightarrow\otop(\overline{X})\longrightarrow\Hom_\cts(\pixbar,\pigmtame)\overset{\sim}{\longrightarrow} \hone$$
where the first map comes from Proposition \ref{base}, and the second map comes from the functoriality of $\pi_1^\et$ along with the composition of any map $\pixbar\rightarrow\pigmbar$ with the quotient $\pigmbar\rightarrow\pigmtame$. With this we get an analogue of [\cite{Voe}, \href{https://www.math.ias.edu/vladimir/sites/math.ias.edu.vladimir/files/Etale_topologies_published.pdf\#page=12}{Proposition 3.2}]:
\begin{lem}\label{commdiag}
    
We have a commutative diagram
$$\begin{tikzcd}
                               & \ohat \arrow[rd, "\chi"]       &        \\
\OO^*(X) \arrow[ru] \arrow[rd] &                                & \hone \\
                               & \otop(X) \arrow[ru, "\chi_\top"'] &       
\end{tikzcd}$$
\end{lem}
\begin{proof}
    Let $X=\spec(A)$. We would like to show that at finite level $n$ prime-to$-p$, the bottom composition sends a unit $u$ to the class corresponding to the \'etale $\mu_n$-torsor $(A\otimes_K\kbar)[T]/(T^n-u)$ (which is invariant under the Galois action because $u\in\OO^*(X)$). After tensoring up to $\kbar$, finite level $n$ sees only the degree$-n$ Kummer cover $\GG_m\overset{x\mapsto x^n}{\longrightarrow}\GG_m$. The degree$-n$ Kummer cover gets sent to $A[T]/(T^n-u)$ along $u_\et$. And the induced map on \'etale fundamental groups gives the desired torsor (it is given by pullback).
\end{proof}

By assumption, $X_0$ is of finite type over $K$ and has a $K$-rational point. So we have a split exact sequence 
$$1\longrightarrow\pi_1^\et(\overline{X_0},\overline{x})\longrightarrow\pi_1^\et({X_0},\overline{x})\longrightarrow\Gamma_K\longrightarrow1$$
By topological invariance, we get a split exact sequence $$1\longrightarrow\pixbar\longrightarrow\pix\longrightarrow\Gamma_K\longrightarrow1$$
which gives an action of $\Gamma_K$ on $\pixbar$. There is also an action of $\Gamma_K$ on $\pigmbar^t$ inherited from the action of $\Gamma_K$ on $\pigmbar$ because $\pigmbar^t$ is a normal subgroup of $\pigmbar$. 
As in characteristic $0$, we have $\operatorname{Im}\chi_\top\subset\hone^{\Gamma_K}$ (as mentioned in the proof of Lemma \ref{commdiag} above).
\begin{rmk}
    Evidently everything we have done in this subsection works if $X$ is a geometrically connected scheme of finite type over $K$ as long as we make it explicit that we are taking prime-to$-p$ limits. We don't specify prime-to-$p$ because perfection kills the $p$ part.
\end{rmk}
\subsection{The Main Theorem} We can now obtain characteristic $p$ analogues of Voevodsky's key propositions.
\begin{prop}\label{2.10}
    Any class $\xi\in\hone^{\Gamma_K}$ gives rise to a map $$f_\xi:X(\overline{K})\longrightarrow \underset{E\supset K}{\varinjlim}\:\underset{p\,\nmid\,n}{\varprojlim}\: E^\times/\:E^{\times n}$$
    where the colimit is taken over finite extensions $E$ of $K$. Let $$i:\overline{K}^\times\longrightarrow\underset{E\supset K}{\varinjlim}\:\underset{p\,\nmid\,n}{\varprojlim}\: E^\times/\:E^{\times n}$$
    be the natural map. Then for any $\overline{y}\in X(\overline{K})$ and $\varphi\in \otop(X)$ such that $\overline{\varphi}(\overline x)=1$ we have $$i(\overline{\varphi}(\overline y))=f_{\chi_\top(\varphi)}(\overline{y})$$
    where $\overline{\varphi}:X(\overline{K})\longrightarrow\GG_m(\overline{K})=\overline{K}^\times$ is the map on geometric points corresponding to $\varphi$.
\end{prop}
\begin{proof}
    Consider the Kummer sequence for $\kbar^\times$
    $$1\longrightarrow\mu_n\longrightarrow\kbar^\times\longrightarrow\kbar^\times\longrightarrow1$$
    which exists for all $n$. Taking the long exact sequence in Galois cohomology, we get $$1\longrightarrow (\kbar^\times)^{\Gamma_K}/(\kbar^{\times n})^{\Gamma_K}\longrightarrow H^1(\Gamma_K,\mu_n)\longrightarrow H^1(\Gamma_K,\kbar^\times)\longrightarrow1$$
  We have $H^1(\Gamma_K,\kbar^\times)=0$ by Hilbert 90 and $(\kbar^\times)^{\Gamma_K}=K^\times$. Hence, $H^1(\Gamma_K,\mu_n)\cong K^\times/K^{\times n}$ and we conclude that $$H^1(\Gamma_K,\pigmtame)\cong H^1(\Gamma_K,\widehat{\ZZ}_{\neq p})\cong \underset{p\,\nmid\,n}{\varprojlim}\: K^\times/\:K^{\times n}\cong H^1(\Gamma_K,\widehat{\ZZ})$$
     where the last isomorphism comes from the fact that $H^1(\Gamma_K,\mu_{p^n})=0$ for all $n$. A class $\xi\in H^1_\et(\overline{X},\widehat{\ZZ}(1))^{\Gamma_K}$ corresponds to a (continuous) $\Gamma_K$-equivariant map $\pixbar\longrightarrow\pigmtame$. So we get a map $$H^1(\Gamma_K,\pixbar)\longrightarrow H^1(\Gamma_K,\pigmtame)$$ Composing with the map $$X(K)\longrightarrow H^1(\Gamma_K,\pixbar)$$ coming from sections of the homotopy exact sequence that are induced by rational points and passing to the colimit over finite extensions $E\supset K$ gives us $f_\xi$.
    By Proposition \ref{perfsq}, we have a commutative diagram $$\begin{tikzcd}
X(\overline{K}) \arrow[rr, "\overline{\varphi}"] \arrow[d]               &  & \overline{K}^\times \arrow[d] \\
{\underset{E\supset K}{\varinjlim}\:H^1(\Gamma_E,\pi_1^\et(\overline{X},\overline{x}))} \arrow[rr] &  & {\underset{E\supset K}{\varinjlim}\:H^1(\Gamma_E,\pi_1^\et(\GG_{m,\overline{K}},1))} 
\end{tikzcd}$$
By composing with  $$ \underset{E\supset K}{\varinjlim}\:H^1(\Gamma_E,\pi_1^\et(\GG_{m,\overline{K}},1))\longrightarrow \underset{E\supset K}{\varinjlim}\:H^1\big(\Gamma_E,\pi_1^\et(\GG_{m,\overline{K}},1)^{t}\big)\cong\underset{E\supset K}{\varinjlim}\:\underset{p\,\nmid\,n}{\varprojlim}\: E^\times/\:E^{\times n}$$
we get a commutative diagram
$$\begin{tikzcd}
X(\overline{K}) \arrow[rr,"\overline{\varphi}"] \arrow[d]               &  & \overline{K}^\times \arrow[d] \\
{\underset{E\supset K}{\varinjlim}\:H^1\big(\Gamma_E,\pixbar\big)} \arrow[rr] &  & {\underset{E\supset K}{\varinjlim}\:\underset{p\,\nmid\,n}{\varprojlim}\: E^\times/\:E^{\times n}} 
\end{tikzcd}$$
The fact that the composite $$\kbar^\times\longrightarrow \underset{E\supset K}{\varinjlim}\:H^1(\Gamma_E,\pi_1^\et(\GG_{m,\overline{K}},1))\longrightarrow \underset{E\supset K}{\varinjlim}\:H^1\big(\Gamma_E,\pi_1^\et(\GG_{m,\overline{K}},1)^{t}\big)\cong\underset{E\supset K}{\varinjlim}\:\underset{p\,\nmid\,n}{\varprojlim}\: E^\times/\:E^{\times n}$$ is equal to $i$ follows from [\cite{Stix}, \href{https://link.springer.com/book/10.1007/978-3-642-30674-7}{Corollary 71}] after restricting to prime-to-$p$ parts. In particular, the map $E^\times\longrightarrow H^1(\Gamma_E,\pigmtame)$ sends $u$ to the prime-to$-p$ projective limit of the mod$-n$ Kummer torsors given by extracting an $n$-th root of $u$. Then $$H^1(\Gamma_E,\pigmtame)\overset{\sim}{\longrightarrow} \underset{p\,\nmid\,n}{\varprojlim}\: E^\times/\:E^{\times n}$$ sends this profinite Kummer torsor to the prime-to$-p$ projective limit of $u\operatorname{mod}E^{\times n}$. We conclude that $i(\overline\varphi(\overline{x}))=f_{\chi_\top(\varphi)}(\overline{x})$ from the second diagram. 
\end{proof}
The following lemma guarantees that our main argument works \'etale-locally on $X$.
\begin{lem}\label{2.13}
    Let $X=X_0^\perf$ for $X_0$ a $K$-scheme. Let $U\longrightarrow X$ be \'etale. Then $U\cong U_0^\perf$ for some $U_0$ \'etale over $X_0$.
\end{lem}
\begin{proof} 
Topological invariance tells us that pullback along $X\longrightarrow X_0$ induces an equivalence $ X_{0,\et}\simeq X_\et$. Let $U\longrightarrow X$ be \'etale. By essential surjectivity, there exists an \'etale $U_0\longrightarrow X_0$ such that  \begin{align*}
    U&\cong U_0\times_{X_0}X\\&\cong(U_0\times_{X_0}X)^\perf\\&\cong U_0^\perf\times_{X}X\\
    &\cong U_0^\perf
\end{align*} where the second isomorphism comes from Lemma \ref{perf} and the third isomorphism comes from the fact that perfection is a limit.
\end{proof}
Now we assume that $\operatorname{trdeg}(K)\geq 1$. The following argument is the only thing that breaks when working over a finite field. We need $K$ to admit a nontrivial valuation.
\begin{prop}[Decompletion]\label{2.11}
    If $\xi\in \hone^{\Gamma_K}\cap\: \operatorname{Im}\chi$ and $\operatorname{Im}f_\xi\subset \operatorname{Im}i$, then $\xi\in\operatorname{Im}\OO^*(X)$.
\end{prop}
\begin{proof}
For easy reference, here is the diagram from Lemma \ref{commdiag}:
$$\begin{tikzcd}
                               & \ohat \arrow[rd, "\chi"]       &        \\
\OO^*(X) \arrow[ru] \arrow[rd] &                                & \hone \\
                               & \otop(X) \arrow[ru, "\chi_\top"'] &       
\end{tikzcd}$$

  Let $\OO^*_1(\xbar)$ be the subgroup of $\OO^*(\xbar)$ consisting of units whose corresponding maps from $\xbar$ to $\GG_m$ send $\overline{x}$ (our fixed basepoint) to $1$. The images in $\hone$ of both groups coincide because we may rescale our profinite Kummer torsor by an element of $\kbar$ without changing it ($\kbar$ already contains all roots of all of its elements). We can now use the construction of $f_\xi$ from Proposition \ref{2.10}. Note that there are no constants in $\OO^*_1(\xbar)$.
    
    Since $X_0$ is of finite type, $\OO^*_1(\overline{X_0})$ is a finitely generated free abelian group. Its perfect closure $\varinjlim_{x\mapsto x^p}\OO^*_1(\overline{X_0})$ is $\OO^*_1(\overline{X})$ which is thus seen to be a finitely generated free $\ZZ[1/p]$-module. Let $g_1,\ldots,g_n$ be its generators. Then $\widehat\OO^*_1(\overline{X})\cong\widehat{\OO}^*_1(\xbar)_\ptp$ is isomorphic to $\widehat{\ZZ}_\ptp^n$ with $\widehat{\ZZ}_\ptp$-module generators $g_1,\ldots,g_n$. An element of $\OO^*_1(\xbar)$ corresponds to a map on geometric points $$\xbar(\kbar)=X(\kbar)\longrightarrow\kbar^\times$$ So every completed unit $\hat{u}\in\widehat{\OO}^*_1(\xbar)$ corresponds to a map \begin{align*}\hat{u}:X(\overline K)&\longrightarrow\underset{E\supset K}{\varinjlim}\:\underset{p\,\nmid\,n}{\varprojlim}\: E^\times/\:E^{\times n}\end{align*}
   We will abuse notation and freely switch between viewing $\hat u$ as a completed unit and a map. We have $\hat u=g_1^{\varepsilon_1}\ldots g_n^{\varepsilon_n}$ with $\varepsilon_i\in\widehat\ZZ_{\neq p}$ for all $i$. And $\hat{u}=f_{\chi(\hat{u})}$ by Proposition \ref{2.10}. Therefore, $\operatorname{Im}f_{\chi(\hat u)}\subset\operatorname{Im}i$ if and only if $\hat{u}(X(\overline K))\subset i(\overline{K}^\times)$. Assuming that $\hat{u}(X(\overline K))\subset i(\overline{K}^\times)$, we would like to show that $\hat u$ lies in the image of $\OO^*(X)$ i.e., that $\varepsilon_i\in\ZZ[1/p]$ for all $i$. We are guaranteed that it would lie in the image of $\OO^*(X)$ and not just $\OO^*(\overline{X})$ because $\xi\in\hone^{\Gamma_K}$ and not just $\hone$.

   Let $E$ be a finite extension of $K$. Since $\operatorname{trdeg}(E)\geq1$, there exists a non-trivial valuation $\nu:E^\times\longrightarrow\ZZ$. It extends to a map $\hat\nu:\underset{p\,\nmid\, n}{\varprojlim}\:E^\times/E^{\times n}\longrightarrow \widehat{\ZZ}_\ptp$. Consider $x_1,\ldots,x_n\in X(E)$. For every $x_j$, we have $$\sum_i\varepsilon_i\hat\nu(g_i(x_j))=\hat\nu(\hat{u}(x_j))$$ 
   Let $$N=\begin{pmatrix}
       \hat\nu(g_1(x_1))&\ldots&\hat\nu(g_n(x_1))\\
       \vdots&\ddots&\vdots\\
       \hat\nu(g_1(x_n))&\ldots&\hat\nu(g_n(x_n))
       
   \end{pmatrix}$$
By Cramer's Rule we have $$\operatorname{det}(N)\varepsilon_i=P_i(\hat\nu(g_i(x_1)),\ldots,\hat\nu(g_i(x_n)),\hat\nu(\hat{u}(x_1)),\ldots,\hat\nu(\hat{u}(x_n)))=:P_i$$
where the right-hand side is some polynomial in $\hat\nu(g_i(x_1)),\ldots,\hat\nu(g_i(x_n)),\hat\nu(\hat{u}(x_1)),\ldots,\hat\nu(\hat{u}(x_n))$. We have $\hat{u}(X(\overline K))\subset i(\overline{K}^\times)$ by assumption. In particular we have $\det(N),\, P_i\in\ZZ[1/p]\subset\widehat{\ZZ}_\ptp$ because $\nu(i(E^\times))\subset\ZZ[1/p]$. If $\varepsilon\in\widehat{\ZZ}_\ptp$ and $n,n\varepsilon\in\ZZ[1/p]$, then $\varepsilon\in\ZZ[1/p]$. So we may conclude that $\varepsilon_i\in\ZZ[1/p]$ if we can show that $\det(N)\neq 0$. To show this, we use induction on $n$. 

For $n=1$, we have $\det(N)=\hat\nu(g_1(x_1))$. Since $g_1$ is non-constant, it is open. In particular, we may choose $x_1$ so that $\hat\nu(g_1(x_1))\neq0$. Assume we may choose $n-1$ $x_j$'s so that the determinant is nonzero. We have $$\det(N)=\sum(-1)^{i+1}\hat\nu(g_i(x_1))D_i$$
where $D_i$ is the determinant of the matrix obtained by deleting the top row and $i$-th column of $N$.
  By assumption, we may find $x_2,\ldots,x_n$ such that $D_i\neq0$ for all $i$. The inverse image of $\det(N)$ along $\hat\nu$ is $$g'=g_1^{D_1}g_2^{-D_2}\ldots g_n^{(-1)^{n+1}D_n}$$
  Since the $g_i$ are generators of $\OO^*_1(\xbar)$ and the $D_i$ are nonzero, $g'$ is non-constant. Hence, there exists some $x_1$ such that $\hat\nu(g'(x_1))\neq0$. We conclude that $\det(N)\neq0$.
\end{proof}
From now on let $K$ be an infinite, AFG field of positive characteristic. As in the characteristic $0$ case, we have a result on the structure theory of Picard groups of finite type $K$-schemes that will allow us to show that the first condition of Proposition \ref{2.11} holds. 
\begin{thm}[\cite{GJRW}, \href{https://arxiv.org/pdf/alg-geom/9410031\#page=26}{Theorem 6.6}]\label{picx}
    Let $X_0$ be a finite type $K$-scheme. Then $\pic(X_0)$ is of the form
    $$(\text{Countably Generated Free Abelian Group})\oplus\bigoplus_\NN\ZZ/p^n\ZZ\oplus(\text{Finite Group})$$
    for some $n$.
\end{thm}
We have everything we need to prove the main theorem.
\begin{thm}\label{main}
    Let $X=X_0^\perf$ be the perfection of a finite type $K$-scheme and let $Y$ be any finite type $K^{p^{-\infty}}-$scheme. Then the natural map
    $$\mor_{K^{p^{-\infty}}}(X,Y)\longrightarrow\mor_{K_\et}^\bullet(X_\et,Y_\et)$$
    is a bijection.
\end{thm}
\begin{proof}
    We already know that the map is injective by Proposition \ref{inj} (noting that $X_0^\perf\cong X_{0,K^{p^{-\infty}}}^\perf$). By Proposition \ref{suff}, it suffices to show that, for all $U\in X_\et$ and all $\varphi\in\otop(U)=\mor_{K_\et}(U_\et,\GG_{m,\et})$, there exists a map $\widetilde{\varphi}:U\longrightarrow \GG_m$ coinciding with $\varphi$ on geometric points. By Lemma \ref{2.13}, the following argument applies to all \'etale opens over $X$. In particular, it suffices to just consider $X$. Pick some $\varphi\in\otop(X)$. We may assume without loss of generality that $\overline\varphi(\overline{x})=1$. Proposition \ref{2.11} tells us that $\chi_\top(\varphi)$ comes from a unit if we can show that $\chi_\top(\varphi)\in H^1_\et(\xbar,\widehat{\ZZ}(1))^{\Gamma_K}\cap\:\operatorname{Im}\chi$ and $\operatorname{Im}f_{\chi_\top(\varphi)}\subset \operatorname{Im}i$. By Proposition \ref{2.10}, the latter condition is satisfied. Hence, $\chi_\top(\varphi)$ comes from a unit if $\operatorname{Im}\chi_\top\subset\operatorname{Im}\chi$. Consider the left exact sequence
    $$1\longrightarrow\ohat\overset{\chi}{\longrightarrow}\hone\longrightarrow \widehat{T}(\pic(\xbar))$$
    Taking Galois-fixed points gives 
    $$1\longrightarrow \widehat{\OO}^*(\xbar)^{\Gamma_K}\overset{\chi}{\longrightarrow} \hone^{\Gamma_K}\longrightarrow\widehat{T}(\pic(X))$$
    We have $\pic(X)\cong\pic(X_0)[1/p]$ by Lemma \ref{picperf}. So we have $\widehat{T}(\pic(X))\cong \widehat{T}(\pic(X_0)[1/p])$. But Theorem \ref{picx} tells us that $\widehat{T}(\pic(X_0)[1/p])=0$. Hence, $$\operatorname{Im}\chi\supset\chi(\widehat{\OO}^*(\xbar)^{\Gamma_K})=\hone^{\Gamma_K}\supset\operatorname{Im}\chi_\top$$
    Let $\widetilde{\varphi}$ be the unit that maps to $\chi_\top(\varphi)$. We have $\chi_\top(\widetilde{\varphi}_\et)=\chi_\top(\varphi)$ by Lemma \ref{commdiag}. So we know that $f_{\chi_\top(\widetilde{\varphi}_\et)}(\overline{x})=f_{\chi_\top(\varphi)}(\overline{x})$ for all $\overline{x}\in X(\kbar)$. Hence, by Proposition \ref{2.10}, we have $i(\overline{\widetilde{\varphi}_\et}(\overline{x}))=i(\overline{\varphi}(\overline{x}))$ for all $\overline{x}\in X(\kbar)$. Thus, since $i$ is injective ($K$ has no infinitely divisible elements), $\overline{\widetilde{\varphi}_\et}(\overline{x})=\widetilde{\varphi}(\overline{x})=\overline{\varphi}(\overline{x})$ for all $\overline{x}\in X(\kbar)$. In particular, $\widetilde{\varphi}$ coincides with $\varphi$ on geometric points.
\end{proof}
\begin{cor}\label{cor}
    Let $X$ be the perfection of a $K$-scheme of finite type. Let $Y=Y_0^\perf$ be the perfection of a $K$-scheme of finite type. Then the natural map $$\mor_{K^{p^{-\infty}}}(X,Y)\longrightarrow\mor_{K_\et}^\bullet(X_\et,Y_\et)$$
    is bijective.
\end{cor}
\begin{proof}
    We have
    \begin{align*}
        \mor_{K^{p^{-\infty}}}(X,Y)&\cong\lim_\Phi \mor_{K^{p^{-\infty}}}(X,Y_{0,K^{p^{-\infty}}})\\
&\cong\lim_{\Phi_\et}\mor_{K_\et}^\bullet(X_\et,(Y_{0,K^{p^{-\infty}}})_\et)\tag{Theorem \ref{main}}\\
        &\cong \mor_{K_\et}^\bullet(X_\et,Y_\et)\tag{Topological Invariance}
    \end{align*}
\end{proof}
\begin{cor}\label{coriso}
    Let $X$ and $Y$ be perfections of schemes of finite type over $K$. Then $X\cong Y$ if and only if $X_\et\simeq Y_\et$ (over $K^{p^{-\infty}}$ and $\spec(K)_\et$ respectively). \qed
\end{cor}
\begin{schk}[Multiplication by $p$ on $\GG_{m}$]\label{schkgm}
    The $p$-th power map $p:\GG_m\overset{x\mapsto x^p}{\longrightarrow}\GG_m$ gives rise to an equivalence on \'etale sites $p^*:\GG_{m,\et}\overset{\sim}{\longrightarrow}\GG_{m,\et}$ whose inverse is not realized by a morphism of schemes. We can understand this concretely by looking at classes. 
    
    Let $\GG_m=\spec(K^{p^{-\infty}}[t^{\pm1}])$. The $p$-th power map is the unit $t^p$. So the corresponding class at finite level $n$ (prime-to$-p$) is the \'etale $\mu_n$-torsor $\spec(\kbar[t^{\pm1},X]/(X^n-t^p))$. As a class in $H^1_\et(\GG_{m,\kbar},\mu_n)\cong\ZZ/n\ZZ$, this corresponds to $p\mod n$. The corresponding element in $H^1_\et(\GG_{m,\kbar},\widehat{\ZZ}_\ptp)\cong\widehat{\ZZ}_\ptp$ is $p$. The torsor $\lim_{\,p\,\nmid\,n} \spec(\kbar[t^{\pm1},X]/(X^n-t))$ corresponds to $1\in\widehat{\ZZ}_\ptp$. And raising $t$ to a power on the torsor side corresponds to multiplication on the class side.
    
    The inverse of $p^*$ should give a map on \'etale fundamental groups that is the inverse of the map on \'etale fundamental groups induced by $p^*$. In particular, the class associated to the inverse of $p^{*}$ should be the inverse of the class associated to $p^*$. This class is $1/p\in\widehat{\ZZ}_\ptp$. It lies both in the image of $\chi$ and the image of $\chi_\top$, but it does not come from a unit. We conclude that decompletion (Proposition \ref{2.11}) cannot go through for generic imperfect schemes.   
    To see the failure of the decompletion argument in the case of $\GG_m$, we note that $\OO^*_1(\GG_{m,\kbar})\cong \ZZ$ is not a free $\ZZ[1/p]$-module. In particular, we would need to show that $\varepsilon_i\in\ZZ$, not $\ZZ[1/p]$. But the analogue of [\cite{Voe}, \href{https://www.math.ias.edu/vladimir/sites/math.ias.edu.vladimir/files/Etale_topologies_published.pdf#page=13}{Lemma 3.1}] must take account of the fact that everything has to be prime-to-p when working in characteristic p. Indeed, the completed unit with image $1/p\in H^1(\GG_{m,\kbar},\widehat{\ZZ}_\ptp)$ is given by $\hat u=t^{1/p}$ (where $1/p$ is understood as a profinite prime-to$-p$ integer), but $1/p\not\in\ZZ$.
\end{schk}
\clearpage


\begin{thebibliography}{Artin}
\bibitem[BGH]{BGH}
\textbf{[BGH]} Clark Barwick, Saul Glasman, and Peter Haine. \emph{Exodromy.} arXiv:1807.03281 [math.AT], 2020. Available online: \url{https://arxiv.org/abs/1807.03281}.

\bibitem[BS]{BS}
\textbf{[BS]} Bhargav Bhatt and Peter Scholze. \emph{Projectivity of the Witt vector affine Grassmannian.} arXiv:1507.06490 [math.AG], 2017. Available online: \url{https://arxiv.org/abs/1507.06490}.

\bibitem[CHW]{CHW}
\textbf{[CHW]} Magnus Carlson, Peter J. Haine, and Sebastian Wolf. \emph{Reconstruction of schemes from their \'etale topoi.} arXiv:2407.19920 [math.AG], 2024. Available online: \url{https://arxiv.org/abs/2407.19920}.

\bibitem[CS]{CS}
\textbf{[CS]} Magnus Carlson and Jakob Stix. \emph{The \'etale topos reconstructs varieties over sub-p-adic fields.} arXiv:2410.22474 [math.AG], 2024. Available online: \url{https://arxiv.org/abs/2410.22474}.

\bibitem[EDP]{EDP}
\textbf{[EDP]} Alexandre Grothendieck. \emph{Esquisse d'un programme}. Geometric Galois Actions, vol. 1, pp. 5–48. With an English translation on pp. 243–283. London Math. Soc. Lecture Note Ser., vol. 242. Cambridge University Press, Cambridge, 1997.

\bibitem[GJRW]{GJRW}
\textbf{[GJRW]} Robert Guralnick, David Jaffe, Wayne Raskind, and Roger Wiegand. \emph{On the Picard group: torsion and the kernel.} arXiv:alg-geom/9410031 [alg-geom], 1994. Available online: \url{https://arxiv.org/abs/alg-geom/9410031}.

\bibitem[LTF]{LTF}
\textbf{[LTF]} Alexander Grothendieck. \emph{Brief an G. Faltings.} Geometric Galois Actions, 1, 49–58. With an English translation on pp. 285–293. London Math. Soc. Lecture Note Ser., 242. Cambridge University Press, Cambridge, 1997. Available online: \url{https://webusers.imj-prg.fr/~leila.schneps/grothendieckcircle/Letters/GtoF.pdf}.

\bibitem[Rush]{Rush}
\textbf{[Rush]} David E. Rush. \emph{Picard groups in abelian group rings.} Journal of Pure and Applied Algebra, vol. 26, no. 1, pp. 101–114, 1982. ISSN: 0022-4049. DOI: \url{https://doi.org/10.1016/0022-4049(82)90032-9}. Available online: \url{https://www.sciencedirect.com/science/article/pii/0022404982900329}.

\bibitem[Rydh]{Rydh}
\textbf{[Rydh]} David Rydh. \emph{Submersions and effective descent of étale morphisms.} Bulletin de la Société mathématique de France, vol. 138, no. 2, pp. 181–230, 2010. ISSN: 2102-622X. DOI: \url{10.24033/bsmf.2588}. Available online: \url{http://dx.doi.org/10.24033/bsmf.2588}.

\bibitem[RZ]{RZ}
\textbf{[RZ]} Ribes, L., Zaleskii, P. A.: \emph{Profinite Groups}, Springer-Verlag, Berlin etc., Ergebnisse 3. Folge 40, 2000. Available online: \url{https://link.springer.com/book/10.1007/978-3-662-04097-3}.

\bibitem[SS]{SS}
\textbf{[SS]} Karl Schwede and Bernard Serbinowski. \emph{Seminormalization package for Macaulay2.} Journal of Software for Algebra and Geometry, vol. 10, no. 1, pp. 1–7, February 2020. ISSN: 1948-7916. DOI: \url{https://doi.org/10.2140/jsag.2020.10.1}. Available online: \url{http://dx.doi.org/10.2140/jsag.2020.10.1}.

\bibitem[Stacks]{Stacks}
\textbf{[Stacks]} The Stacks Project. \emph{Stacks Project.} Available online: \url{https://stacks.math.columbia.edu/}.

\bibitem[Stix]{Stix}
\textbf{[Stix]} Jakob Stix. \emph{Rational points and arithmetic of fundamental groups.} Lecture Notes in Mathematics, vol. 2054, Springer, Heidelberg, 2013. Evidence for the section conjecture. MR 2977471, DOI 10.1007/978-3-642-30674-7. Available online: \url{https://link.springer.com/book/10.1007/978-3-642-30674-7}.

\bibitem[Voe]{Voe}
\textbf{[Voe]} V. A. Voevodsky. \emph{Etale topologies of schemes over fields of finite type over $\mathbb{Q}$.} Izv. Akad. Nauk SSSR Ser. Mat., vol. 54, pp. 1155--1167, 1990. ISSN: 0373-2436. Available online: \url{https://www.math.ias.edu/vladimir/sites/math.ias.edu.vladimir/files/Etale_topologies_published.pdf}.

\bibitem[Wei]{Wei}
\textbf{[Wei]} Charles A. Weibel, \emph{The K-book: An Introduction to Algebraic K-theory}, Graduate Studies in Mathematics, vol.~145, American Mathematical Society, 2013. ISBN: 978-0-8218-9132-2.
\end{thebibliography}
\end{document}